\documentclass[a4paper,11pt,english]{smfart}

\usepackage[english]{babel}
\usepackage[utf8]{inputenc}
\usepackage[T1]{fontenc}
\usepackage{lmodern}
\usepackage{amsmath}
\usepackage{amssymb}
\usepackage{amsthm}
\usepackage{mathtools}
\usepackage{array}
\usepackage{braket}
\usepackage{tikz}
\usepackage{todonotes}
\usepackage{color}
\definecolor{shadecolor}{gray}{0.90}
\RequirePackage{calrsfs}
\DeclareSymbolFont{rsfscript}{OMS}{rsfs}{m}{b}
\DeclareSymbolFontAlphabet{\mathrsfs}{rsfscript}

\usepackage{bm}

\usepackage{hyperref}

\usepackage{fancyvrb}

\usepackage{lscape}

\usepackage[final]{showlabels}

\usepackage[leftbars]{changebar}

\usepackage{palatino}

\usetikzlibrary{cd}

\numberwithin{equation}{section}

\usepackage[utopia,expert]{mathdesign}

\def\arxivprefix{https://arxiv.org}

\theoremstyle{plain}
\newtheorem{theoreme}{Theorem}[section]
\newtheorem{proposition}[theoreme]{Proposition}
\newtheorem{lemma}[theoreme]{Lemma}
\newtheorem{corollary}[theoreme]{Corollary}
\newtheorem{definition}[theoreme]{Definition}

\theoremstyle{definition}
\newtheorem{exemple}[theoreme]{Example}

\theoremstyle{remark}
\newtheorem*{remark}{Remark}

\DeclareMathOperator{\End}{End}
\DeclareMathOperator{\Hom}{Hom}

\DeclareMathOperator{\id}{id}
\DeclareMathOperator{\Id}{Id}
\DeclareMathOperator{\Tr}{Tr}

\DeclareMathOperator{\Irr}{Irr}
\DeclareMathOperator{\ev}{ev}
\DeclareMathOperator{\coev}{coev}
\DeclareMathOperator{\Gr}{Gr}
\DeclareMathOperator{\sGr}{sGr}
\DeclareMathOperator{\Mat}{Mat}
\DeclareMathOperator{\sVect}{sVect}
\DeclareMathOperator{\sdim}{sdim}
\DeclareMathOperator{\rev}{rev}
\DeclareMathOperator{\s}{s}
\DeclareMathOperator{\bid}{Bid}
\DeclareMathOperator{\sbid}{sBid}
\DeclareMathOperator{\Ve}{Vec}

\newcommand*{\fonctionNom}[5]{
  #1 \colon \left\{
  \begin{array}{ccc}
    #2 & \longrightarrow & #3 \\
    #4 & \longmapsto & #5
  \end{array}
  \right.
}

\def\boitegrise#1#2{\begin{centerline}{\fcolorbox{black}{shadecolor}{~
    \begin{minipage}[t]{#2}{\vphantom{~}#1\vphantom{$A_{\displaystyle{A_A}}$}}
            \end{minipage}~}}\end{centerline}\medskip}

\newcommand{\e}{\bm{\varepsilon}}
\newcommand{\1}{\bm{1}}

\title{Slightly degenerate categories and $\mathbb{Z}$-modular data}
\author{Abel Lacabanne}
\email{abel.lacabanne@normalesup.org}  
\address{Institut de recherche en mathématique et physique,
Chemin du Cyclotron 2, boîte L7.01.01,
1348 \sc Louvain-la-Neuve,
\sc Belgique}
\date{\today}

\begin{document}
\maketitle 

\pagestyle{myheadings}
\markboth{\sc A. Lacabanne}{\sc Slightly degenerate categories and $\mathbb{Z}$-modular data}

\section{Introduction}
An $\mathbb{N}$-modular datum over $\mathbb{C}$ is a finite set $I$, a distinguished element $i_0\in I$, a square matrix
$S\in\mathcal{M}_{I}(\mathbb{C})$, a diagonal matrix
$T\in\mathcal{M}_{I}(\mathbb{C})$ such that
\begin{enumerate}
\item for any $i\in I$, $S_{i_0,i}\neq 0$,
\item $S$ is unitary, symmetric and $S^4=\Id$, $(ST)^3=\lambda \Id$
  and $[S^2,T]=\Id$,
\item for any $i,j,k\in I$,
  \[
    N_{i,j}^k =
    \sum_{l\in I}\frac{S_{i,l}S_{j,l}\overline{S_{k,l}}}{S_{i_0,l}}
  \]
  belongs to $\mathbb{N}$.
\end{enumerate}

Modular categories naturally give rise to modular data \cite[Section
8.16]{egno} and given a modular datum, one can ask the question of
finding a modular category with this modular datum. In
\cite{exotic}, Lusztig gives a slightly more restrictive definition
of modular datum, and associate a modular datum to each dihedral
group.

In order to generalize Lusztig's work for imprimitive complex reflection
groups, Malle \cite{unip_malle} defines a modular datum, but such that the
integers $N_{i,j}^k$ are in $\mathbb{Z}$, which we will call
$\mathbb{Z}$-modular datum. We need another categorical framework to
understand these data, since a
modular category always defines an $\mathbb{N}$-modular datum: the
integers $N_{i,j}^k$ are the multiplicities of the object $k$ in the
tensor product $i\otimes j$.

In \cite{double_taft} Bonnafé and Rouquier gave a categorification of the Malle
$\mathbb{Z}$-modular datum associated with cyclic groups, by
constructing a tensor triangulated category with extra structure.

In this article, we explain how slightly degenerate braided pivotal
fusion categories \cite[Definition 2.5]{eno-weakly} give rise to $\mathbb{Z}$-modular
datum. This is also related to braided pivotal superfusion
categories \cite{monoidalsupercat}. Note that there exist two pivotal structures on
supervector spaces, one of which is unitary (and therefore all simple objects have positive dimension), and the other is not (the two simple objects
are of dimension $1$ and $-1$). In \cite{16fold}, the authors consider super-modular categories (unitary modular categories with symmetric center equal to $\sVect$) but they do not get a modular datum from such a category. We will show that with the non-unitary structure, a slightly degenerate category gives rise to a $\mathbb{Z}$-modular datum.

As an application, we will reinterpret the example of Bonnafé and Rouquier in
this setting of slightly degenerate categories. This approach will
be generalized in \cite{double_g}.

This paper is organized as follows. Section \ref{sec:cat} extends
well-known results on modular categories to pivotal braided fusion
categories which are not necessarily spherical: similar to the
modular case, we construct an action of the group $SL_2(\mathbb{Z})$ on
the Grothendieck group of a nondegenerate braided pivotal fusion
category and we show that the twist are always roots of unity. Section \ref{sec:degfus} introduces slightly degenerate categories, and we
explain how such a category gives rise to a $\mathbb{Z}$-modular
datum. We introduce in Section \ref{sec:supermon} the notion of a
supercategory \cite{monoidalsupercat}, and explain how to produce a supercategory from a
slightly degenerate category with extra structure. Finally, Section
\ref{sec:appli} is devoted to the example of Bonnafé-Rouquier
in the setting of slightly degenerate categories.


\subsection*{Acknowledgements}
  I warmly thank my advisor C. Bonnaf\'e for many fruitful discussions
  and his constant support. The paper is partially based upon work
  supported by the NSF under grant DMS-1440140 while the author was in
  residence at the Mathematical Sciences Research Institute in
  Berkeley, California, during the Spring 2018 semester.

\section{Nondegenerate braided pivotal fusion categories}
\label{sec:cat}

 Let $\Bbbk$ be an algebraically closed field of characteristic $0$
 and $\mathcal{C}$ be a tensor category over $\Bbbk$, as defined in
 \cite[Definition 4.1.1]{egno}: $\mathcal{C}$ is a locally finite
 $\Bbbk$-linear abelian rigid monoidal category (with unit object denoted by
 $\1$) such that the bifunctor $\otimes \colon
 \mathcal{C}\times\mathcal{C}\rightarrow\mathcal{C}$ is
 $\Bbbk$-bilinear on morphisms and
 $\End_{\mathcal{C}}(\1)=\Bbbk$. We denote by $\alpha_{X,Y,Z}\colon
 (X\otimes Y)\otimes Z \rightarrow X\otimes(Y\otimes Z)$ the
 associativity constraint, but will often omit it. The left (resp. right) dual of an object $X\in\mathcal{C}$ is denoted by $X^*$ (resp. ${}^*X$) with evaluation and coevaluation morphism
\[
  \ev_X\colon X^*\otimes X\rightarrow\1 \qquad\mathrm{and}\qquad \coev_X\colon\1\rightarrow X\otimes X^*
\]
(resp.
\[
  \ev'_X\colon X\otimes {}^*X\rightarrow\1 \qquad\mathrm{and}\qquad \coev'_X\colon\1\rightarrow {}^*X\otimes X)
\]
such that the following compositions are identities
  \[
    \begin{tikzcd}[column sep=huge]
      X \arrow{r}{\coev_X \otimes \id_X}&X\otimes X^*\otimes X
      \arrow{r}{\id_X\otimes\ev_X} &X,
  \end{tikzcd}
\]
\[
    \begin{tikzcd}[column sep=huge]
      X^* \arrow{r}{\id_{X^*}\otimes\coev_X}&X^*\otimes X\otimes X^*
      \arrow{r}{\ev_x\otimes\id_{X^*}} &X^*.
  \end{tikzcd}
\]

One can define, for any $f\colon X\rightarrow Y$ with $X$ and $Y$
having left duals, the left dual of $f$
as the map $f^*\colon Y^*\rightarrow X^*$ given by the composition 
\[
  \begin{tikzcd}
    Y^* \ar{r}{\id_{Y^*}\otimes \coev_X} &[3em] Y^*\otimes X \otimes X^*
    \ar{r}{\id_{Y^*}\otimes f \otimes \id_{X^*}} &[3em] Y^*\otimes Y
    \otimes X^* \ar{r}{\ev_Y\otimes\id_{X^*}} &[3em] X^*,
  \end{tikzcd}
\]
and similarly there exists the right dual of a map.

We assume that $\mathcal{C}$ is equipped with a pivotal structure \cite[Definition 4.7.8]{egno}: there is an isomorphism of monoidal functors $a_X\colon X\rightarrow X^{**}$. For $f\in\Hom_{\mathcal{C}}(X,X)$, we can define two traces that are elements of $\End_{\mathcal{C}}(1)=\Bbbk$. The right quantum trace $\Tr^R_X(f)$ is given by the composition
\[
\begin{tikzcd}
  \1 \arrow{r}{\coev_X} &[1em] X\otimes X^* \arrow{r}{(a_X\circ f)\otimes \id_{X^*}} &[4em] X^{**}\otimes X^* \arrow{r}{\ev_{X^*}} &\1,
\end{tikzcd}
\]
and the left quantum trace $\Tr^L_X(f)$ is given by the composition
\[
  \begin{tikzcd}
  \1 \arrow{r}{\coev_{X^*}} &[1em] X^*\otimes X^{**} \arrow{r}{\id_{X^*}\otimes(f\circ a_{X}^{-1})} &[4em] X^*\otimes X \arrow{r}{\ev_{X}} &\1.
\end{tikzcd}
\]
It is well known that for any $f\in\End_{\mathcal{C}}(X)$, $\Tr^R_{X^*}(f^*)=\Tr^L_X(f)$. We also define the partial traces of $f\in\End_{\mathcal{C}}(X\otimes Y)$ by
\[
  \begin{tikzcd}
    \id_X\otimes\Tr^R_Y(f)\colon X \ar{r}{\coev_Y} &[1em] X\otimes Y\otimes Y^* \ar{r}{((\id_X\otimes a_Y)\circ f)\otimes\id_{Y}} &[6em] X\otimes Y^{**}\otimes Y \ar{r}{\ev_Y} & X
  \end{tikzcd}
\]
and
\[
  \begin{tikzcd}
    \Tr_X^L\otimes\id_Y(f)\colon Y \ar{r}{\coev_{X^*}} &[1em] X^*\otimes X^{**}\otimes Y \ar{r}{\id_{X^*}\otimes(f\circ(a_X^{-1}\otimes\id_Y))} &[6em] X^*\otimes X\otimes Y \ar{r}{\ev_X} & Y.
  \end{tikzcd}
\]
Finally, denote by $\Tr^{L,R}_{X\otimes Y}(f)$ the endomorphism of $\1$ given by $\Tr_X^L(\id_X\otimes\Tr^R_Y(f))$ which is then equal to $\Tr_Y^R(\Tr_X^L\otimes\id_Y(f))$.

\begin{lemma}
  \label{lem:trace-mixed-dual}
  Let $\mathcal{C}$ be a pivotal rigid monoidal category, $X$ and $Y$ two objects of $\mathcal{C}$ and $f\in\Hom_{\mathcal{C}}(X\otimes Y,X\otimes Y)$. Then $\Tr^{L,R}_{Y^*\otimes X^*}(f^*) = \Tr^{L,R}_{X\otimes Y}(f)$.
\end{lemma}

\begin{proof}
  In order to simplify notations, we omit tensor product signs between objects. We denote by $a$ the pivotal structure. By definition:
  \begin{align*}
    \id_{Y^*}\otimes\Tr^R_{X^*}(f^*)= (\id_{Y^*}\otimes\ev_{X^{**}})
    \circ (\id_{Y^*}\otimes a_{X^*}\otimes \id_{X^{**}}) \circ (f^*\otimes \id_{X^{**}})
    \circ (\id_{Y^*}\otimes \coev_{X^*}).
  \end{align*}

  Using the definition of $f^*$ and the functoriality of the tensor product, we obtain:
  \begin{align*}
    \id_{Y^*}\otimes\Tr^R_{X^*}(f^*) &= (\ev_Y\otimes \id_{Y^*}) 
    \circ (\id_{Y^*}\otimes \ev_X \otimes \id_{YY^*})
    \circ (\id_{Y^*X^*}\otimes f \otimes \id_{Y^*})
    \circ (\id_{Y^*X^*X}\otimes \coev_Y) \\
    &\circ (\id_{Y^*X^*}\otimes((\id_X\otimes \ev_{X^{**}}) \circ (\id_X\otimes a_{X^*}\otimes \id_{X^{**}}) \circ (\coev_X\otimes \id_{X^{**}})))
    \circ \id_{Y^*}\otimes\coev_{X^*}.
  \end{align*}
  But $a_{X^*}=(a_X^{-1})^*$ (\emph{c.f.} \cite[Exercise 4.7.9]{egno}) and therefore $(\id_X\otimes a_{X^*}) \circ \coev_X = (a_X^{-1}\otimes \id_{X^{***}}) \circ \coev_{X^{**}}$. Using the definition of the duality, one has $(\id_X\otimes \ev_{X^{**}}) \circ (\id_X\otimes a_{X^*}\otimes \id_{X^{**}}) \circ (\coev_X\otimes \id_{X^{**}}) = a_X^{-1}$ and consequently 
  \[
  \id_{Y^*}\otimes\Tr^R_{X^*}(f^*) = (\ev_Y\otimes \id_{Y^*}) \circ (\id_Y^* \otimes (\Tr^{L}_X\otimes \id_Y)(f) \otimes \id_{Y^*}) \circ (\id_{Y^*}\otimes \coev_Y) = \left[(\Tr^{L}_X\otimes \id_Y)(f)\right]^*.
  \]
  Finally,
  \[
  \Tr^{L,R}_{Y^*\otimes X^*}(f^*) = \Tr^{L}_{Y^*}\left(\left[(\Tr^{L}_X\otimes \id_Y)(f)\right]^*\right) = \Tr^R_{Y}((\Tr^{L}_X\otimes \id_Y)(f))=\Tr^{L,R}_{X\otimes Y}(f),
  \]
as expected.
\end{proof}

The \emph{left} and \emph{right} \emph{quantum dimensions} are
\[
  \dim^L(X):=\Tr_X^L(\id_X) \qquad \mathrm{and}\qquad \dim^R(X):=\Tr_X^R(\id_X)
\]
which therefore satisfy $\dim^R(X^*)=\dim^L(X)$. Define the squared norm of an object $X$ by
\[
  |X|^2:=\dim^R(X)\dim^L(X)=\dim^R(X)\dim^R(X^*).
\]
It is a totally positive number if $X$ is simple \cite[Proposition 7.21.14]{egno}: for any embedding $\iota$ of the subfield $\Bbbk_{\mathrm{alg}}$ of algebraic elements of $\Bbbk$ in $\mathbb{C}$, one has $\iota(\lvert X \rvert^2) >0$. The dimension of the category $\mathcal{C}$ is
\[
  \dim(\mathcal{C}) := \sum_{X\in\Irr(\mathcal{C})}\lvert X \rvert^2,
\]
where $\Irr(\mathcal{C})$ denotes the set of isomorphism classes of simple objects of $\mathcal{C}$.

We further assume that $\mathcal{C}$ is braided: there exists a family of binatural isomorphisms $c_{X,Y}\colon X\otimes Y \rightarrow Y\otimes X$ such that the hexagon axioms are satisfied \cite[Definition 8.1.1]{egno}.

For a rigid braided tensor category, there exists a natural
isomorphism $u_X \colon X \rightarrow X^{**}$, called the Drinfeld
morphism, defined as the composition
\[
  \begin{tikzcd}
    X\ar{r}{\coev_{X^*}} & X\otimes X^* \otimes
    X^{**}\ar{r}{c_{X,X^*}} & X^*\otimes X\otimes
    X^{**}\ar{r}{\ev_X} & X^{**}.
  \end{tikzcd}
\]
It satisfies for all $X,Y\in\mathcal{C}$,
\[
  u_{X}\otimes u_{Y}= u_{X\otimes Y}\circ c_{Y,X}\circ c_{X,Y}.
\]
Note that we have suppressed the isomorphism $(X\otimes Y)^{**}\simeq X^{**}\otimes Y^{**}$. To give a pivotal structure $a$ on $\mathcal{C}$ is therefore equivalent to give a \emph{twist} on $\mathcal{C}$, which is a natural isomorphism $\theta_X \colon X \rightarrow X$ satisfying for all $X,Y\in\mathcal{C}$
\[
  \theta_{X\otimes Y}=(\theta_X \otimes \theta_Y) \circ c_{Y,X}\circ c_{X,Y}.
\]
The pivotal structure and the twist are related by $a=u\theta$. We will often endow the braided pivotal category $\mathcal{C}$ with the twist given by the pivotal structure.

\begin{remark}
  \label{rk:twist-pivotal}
  There are two ways of identifying twist and pivotal structures \cite[Appendix A.2]{henriques-penneys-tener}. We will always use the one described above.
\end{remark}

\subsection{Semisimplification}
\label{sec:sscat}

We recall the procedure of semisimplification for pivotal categories
(which are not necessarily spherical) which is given in \cite{etingof_ostrik_ss}. Let $\mathcal{C}$ be a braided pivotal tensor category over $\Bbbk$. Denote by $a_X \colon X \rightarrow X^{**}$ the pivotal structure. 

A morphism $f\in\Hom_{\mathcal{C}}(X,Y)$ is said to be \emph{left}
(resp. \emph{right}) \emph{negligible} if for all
$g\in\Hom_{\mathcal{C}}(Y,X)$ one has $\Tr^L_X(g\circ f) = 0$
(resp. $\Tr^R_X(g\circ f) = 0$). An application of \cite[Proposition
1.5.1]{bruguiere_sln} shows that the notions of left and right
negligible morphisms coincide because $\mathcal{C}$ is
braided. Therefore the left quantum dimension of an object is zero if
and only if its right quantum dimension is zero: the assumption (2) of
\cite[Theorem 2.6]{etingof_ostrik_ss} is satisfied. We
then denote by $\Hom_{\mathrm{negl}}(X,Y)$ the subspace of negligible
morphisms. Define a category $\mathcal{C}^{\text{ss}}$ with the same objects as $\mathcal{C}$ and $\Hom_{\mathcal{C}^{\text{ss}}}(X,Y)=\Hom_{\mathcal{C}}(X,Y)/\Hom_{\mathrm{negl}}(X,Y)$.

\begin{proposition}[{\cite[Theorem 2.6]{etingof_ostrik_ss}}]
  Let $\mathcal{C}$ be a braided pivotal tensor category. The category $\mathcal{C}^{\text{ss}}$ is a semisimple braided pivotal tensor category whose simple objects are the indecomposable objects of $\mathcal{C}$ with non-zero right quantum dimension.
\end{proposition}

\subsection{S-matrices and symmetric center}
\label{sec:smat}

The following definition is due to Müger \cite[Definition 2.9]{muger_structure}.

\begin{definition}
  The symmetric center $\mathcal{Z}_{\mathrm{sym}}(\mathcal{C})$ of a braided monoidal category $\mathcal{C}$ is the full subcategory of $\mathcal{C}$ with objects $X$ such that
\[
  \forall Y \in \mathcal{C}, c_{Y,X}\circ c_{X,Y} = \id_{X\otimes Y}.
\]

  We say that $\mathcal{C}$ is nondegenerate if $\1$ is the unique simple object in $\mathcal{Z}_{\mathrm{sym}}(\mathcal{C})$. 
\end{definition}

\boitegrise{{\bf{Hypothesis and notations.}}{\emph{We suppose until
      the end of this Section that the category $\mathcal{C}$ is a
      braided pivotal fusion category. Denote by $\Irr(\mathcal{C})$
      the set of isomorphism classes of simple objects in
      $\mathcal{C}$ and by $\Gr(\mathcal{C})$ its Grothendieck ring
      which admits $([X])_{X\in\Irr(\mathcal{C})}$ as a basis. For
      $X,Y,Z\in\Irr(\mathcal{C})$, we denote by $N_{X,Y}^Z$ the
      multiplicity of $Z$ in the tensor product $X\otimes Y$. Then
      $\Gr(\mathcal{C})$ is a free $\mathbb{Z}$-algebra with basis
      $\Irr(\mathcal{C})$ and the structure constants are given by the
      positive integers $N_{X,Y}^Z$:
      \[
        [X]\cdot [Y] = \sum_{Z\in\Irr(\mathcal{C})}N_{X,Y}^Z[Z].
      \]}}}{0.8\textwidth}

For any simple object $X$, its left and right quantum dimensions are non-zero \cite[Proposition 4.8.4]{egno}. If $X$ and $Y$ are objects of $\mathcal{C}$, we set
\[
  s_{X,Y}^R:=(\id_X\otimes\Tr_Y^R)(c_{Y,X}\circ c_{X,Y}) \in \End_{\mathcal{C}}(X)
\]
and
\[
  s_{X,Y}^L:=(\Tr^L_X\otimes\id_Y)(c_{Y,X}\circ c_{X,Y}) \in \End_{\mathcal{C}}(Y).
\]
These induce two morphisms of abelian groups
\[
  \fonctionNom{s_X^R}{\Gr(\mathcal{C})}{\End_{\mathcal{C}}(X)}{Y}{s^R_{X,Y}} \qquad\mathrm{and}\qquad \fonctionNom{s_Y^L}{\Gr(\mathcal{C})}{\End_{\mathcal{C}}(Y)}{X}{s^L_{X,Y}}.
\]

\begin{proposition}[{\cite[Proposition 8.3.11]{egno}}]
  Let $\mathcal{C}$ be a braided pivotal fusion category. If $X\in\mathcal{C}$ is simple then $s_X^R\colon\Gr(\mathcal{C}) \rightarrow \Bbbk$ and $s_X^L\colon\Gr(\mathcal{C}) \rightarrow \Bbbk$ are morphisms of rings.
\end{proposition}

We now consider the matrices $S^{R,R}$, $S^{L,R}$ and $S^{L,L}$ in $\Mat_{\Irr(\mathcal{C})}(\Bbbk)$ defined by
\begin{align*}
  S^{R,R}_{X,Y} &:= \Tr^R_{X\otimes Y}(c_{Y,X}\circ c_{X,Y}) = \Tr^R_{X}(s^R_{X,Y}),\\
  S^{L,L}_{X,Y} &:= \Tr^L_{X\otimes Y}(c_{Y,X}\circ c_{X,Y}) = \Tr^L_{Y}(s^L_{X,Y}),\\
  S^{L,R}_{X,Y} &:= \Tr^{L,R}_{X\otimes Y}(c_{Y,X}\circ c_{X,Y}) = \Tr^R_{Y}(s^L_{X,Y}) = \Tr^L_{X}(s^R_{X,Y}).
\end{align*}

These three matrices are related as follow
\[
  \frac{\dim^L(X)}{\dim^R(X)}S^{R,R}_{X,Y}=S^{L,R}_{X,Y}=\frac{\dim^R(Y)}{\dim^L(Y)}S^{L,L}_{X,Y}.
\]

\begin{remark}
  The matrices $S^{R,R}$ and $S^{L,L}$ are symmetric but $S^{R,L}$ is not
  in general; if the pivotal structure is spherical, these three
  matrices are equal.

  As for any $f\in\Hom_{\mathcal{C}}(X,X)$ we have
  $\Tr^R_X(f)=\Tr_{X^*}^L(f^*)$, the following relations are
  satisfied:
  \begin{align*}
    S_{X^*,Y^*}^{R,R}&=S_{X,Y}^{L,L},\\
    S_{X^*,Y^*}^{L,R}&=S_{Y,X}^{L,R}.
  \end{align*}
  Hence the matrix $(S^{L,R}_{X,Y^*})_{X,Y\in\Irr(\mathcal{C})}$ is
  symmetric.
\end{remark}

It is clear that if $X \in \Irr(\mathcal{Z}_{\mathrm{sym}}(\mathcal{C}))$ then for all $Y \in \Irr(\mathcal{C})$ we have $S_{X,Y}^{?,?'}=\dim^?(X)\dim^{?'}(Y)$, for $(?,?')\in\{(R,R),(L,R),(L,L)\}$.

\begin{proposition}[{\cite[Proposition 8.20.5]{egno}}]
  \label{prop:smat_sym}
  Let $\mathcal{C}$ be a braided pivotal fusion category and $X$ be a simple object in $\mathcal{C}$. The following are equivalent:
  \begin{enumerate}
  \item $X\in\mathcal{Z}_{\mathrm{sym}}(\mathcal{C})$,
  \item for all $Y \in \Irr(\mathcal{C})$ we have $S_{X,Y}^{L,R}=\dim^L(X)\dim^{R}(Y)$,
  \item for all $Y \in \Irr(\mathcal{C})$ we have $S_{X,Y}^{R,R}=\dim^R(X)\dim^{R}(Y)$,
  \item for all $Y \in \Irr(\mathcal{C})$ we have $S_{X,Y}^{L,L}=\dim^L(X)\dim^{L}(Y)$.
  \end{enumerate}
\end{proposition}

The category $\mathcal{C}$ can be endowed with another braiding,
namely the reverse braiding. We denote it by $c^{\rev}$ and it
is defined by $c_{X,Y}^{\rev}=c_{Y,X}^{-1}$. We denote by
$\mathcal{C}^{\rev}$ the category $\mathcal{C}$ equipped with the
reverse braiding, and by $S^{\rev,R,R}$, $S^{\rev,L,L}$ and
$S^{\rev,L,R}$ the corresponding $S$-matrices. Note that we use the
same pivotal structure on $\mathcal{C}$ and $\mathcal{C}^{\rev}$ for
the computation of the traces. 

\begin{proposition}
  \label{prop:rev-S-mat}
  Let $\mathcal{C}$ a braided pivotal fusion category. Then for any
  $X$ and $Y$ simple objects we have
  \[
    S^{\rev,L,R}_{X,Y}=S^{R,R}_{Y,X^*}.
  \]
\end{proposition}

\begin{proof}
  We start with a lemma, which is a direct consequence of
  \cite[Proposition 8.9.1]{egno}.

  \begin{lemma}
    \label{lem:cup-cap-braid}
    Let $\mathcal{C}$ be a braided rigid tensor category. Then for every
    $X$ and $Y$ objects in $\mathcal{C}$ we have:
    \begin{enumerate}
    \item
      $(\ev_X\otimes\id_Y)\circ(\id_{X^*}\otimes
      c_{Y,X})=(\id_Y\otimes\ev_X)\circ(c_{Y,X^*}^{-1}\otimes\id_X)$,
    \item
      $(c_{X,Y}\otimes \id_{Y^*})\circ (\id_X\otimes\coev_Y) =
      (\id_{Y}\otimes c_{X,Y^*}^{-1})\circ (\coev_Y\otimes\id_X)$.
    \end{enumerate}
  \end{lemma}
  
  Now, by definition
  \begin{multline*}
    S^{\rev,L,R}_{X,Y}=(\ev_X\otimes\ev_{Y^*})
    \circ(\id_{X^*\otimes X}\otimes a_Y\otimes\id_{Y^*})
    \circ(\id_{X^*}\otimes c_{X,Y}^{-1}\otimes \id_{Y^*})\\
    \circ(\id_{X^*}\otimes c_{Y,X}^{-1}\otimes \id_{Y^*})
    \circ(\id_{X^*}\otimes a_X^{-1}\otimes\id_{Y\otimes Y^*})
    \circ(\coev_{X^*}\otimes\coev_{Y}).
  \end{multline*}
  By naturality of the braiding and using
  Lemma \ref{lem:cup-cap-braid} we have
  \begin{align*}
    (\ev_X\otimes\id_{Y^{**}})
    \circ(\id_{X^*\otimes X}\otimes a_Y)
    \circ(\id_{X^*}\otimes c_{X,Y}^{-1})
    &=(\ev_X\otimes\id_{Y^{**}})
      \circ(\id_{X^*}\otimes c_{X,Y^{**}}^{-1})
      \circ(\id_{X^*}\otimes a_Y\otimes \id_{X})\\
    &=(\id_{Y^{**}}\otimes\ev_X)
      \circ(c_{X^*,Y^{**}}\otimes\id_X)
      \circ(\id_{X^*}\otimes a_Y\otimes \id_{X}).
  \end{align*}
 Similarly,
  \[
    (\id_{X^{*}}\otimes c_{Y,X}^{-1})
    \circ(\id_{X^{*}}\otimes a_X^{-1}\otimes \id_Y)
    \circ(\coev_{X^*}\otimes \id_Y)
    =
    (\id_{X^*\otimes Y}\otimes a_X^{-1})
    \circ(c_{Y,X^*}\otimes \id_{X^{**}})
    \circ(\id_Y\otimes\coev_{X^*}).
  \]
  Therefore
  \begin{align*}
    S^{\rev,L,R}_{X,Y}
    &=
      \ev_{X\otimes Y^*}
      \circ(c_{X^*,Y^{**}}\otimes\id_{X\otimes Y^*})
      \circ(\id_{X^*}\otimes a_Y\otimes a_X^{-1}\otimes\id_{Y^*})
      \circ(c_{Y,X^*}\otimes \id_{X^{**}\otimes Y^*})
      \circ\coev_{Y\otimes X^*}\\
    &=
      \ev_{X\otimes Y^*}
      \circ(a_Y\otimes \id_X^*\otimes a_X^{-1}\otimes\id_{Y^*})
      \circ(c_{X^*,Y}\otimes\id_{X^{**}\otimes Y^*})
      \circ(c_{Y,X^*}\otimes \id_{X^{**}\otimes Y^*})
      \circ\coev_{Y\otimes X^*}.
  \end{align*}
Finally, using that $(a_X^{-1})^*=a_{X^*}$ and that for any $f\colon
W\rightarrow Z$ we have $\ev_W\circ(\id_W\otimes f^*) =
\ev_Z\circ(f\otimes\id_{Z^*})$, we obtain $S^{\rev,L,R}_{X,Y}=S^{R,R}_{Y,X^*}$.
\end{proof}

\subsection{Twists and Gauss sums}
\label{subsec:twist-gauss}

We now suppose that the category $\mathcal{C}$ is equipped with a twist $\tilde{\theta}$, and the twist associated to the pivotal structure is denoted by $\theta$. On a simple object $X$, the twist is multiplication by a scalar, and we will identify $\tilde{\theta}_X$ with this scalar.

\begin{proposition}
  \label{prop:twist-dim}
  Let $\mathcal{C}$ be a braided pivotal fusion category. We consider $\theta$ the twist associated to the pivotal structure. For any simple object $X$, we have $\theta_{X^*}\dim^R(X)=\theta_Xdim^L(X)$.
\end{proposition}

\begin{proof}
  Similarly to \cite[Proposition 8.10.14]{egno}, we have $\dim^R(X)=\theta_X \ev_X\circ c_{X,X^*} \circ \coev_X$. We show that $\Tr^{L,R}_{X\otimes X}(c_{X,X}^{-1})=\ev_X\circ c_{X,X^*} \circ \coev_X$. By definition,
\begin{multline*}
  \Tr^{L,R}_{X\otimes X}(c_{X,X}^{-1})= (\ev_X\otimes \ev_{X^*}) \circ (\id_{X^*\otimes X}\otimes a_X\otimes \id_{X^*}) \circ (\id_{X^*}\otimes c_{X,X}^{-1} \otimes \id_{X^*}) \\\circ (\id_{X^*}\otimes a_X^{-1}\otimes \id_{X\otimes X^*})\circ (\coev_{X^*}\otimes \coev_X). 
\end{multline*}
  Using the naturality of the braiding, we get rid of the pivotal structure: 
\[
   \Tr^{L,R}_{X\otimes X}(c_{X,X}^{-1})= (\ev_X\otimes \ev_{X^*}) \circ (\id_{X^*}\otimes c_{X,X^{**}}^{-1} \otimes \id_{X^*}) \circ (\coev_{X^*}\otimes \coev_X).
\]
From the Lemma \ref{lem:cup-cap-braid} we deduce that $(\id_X\otimes \ev_{X^*})\circ(c_{X,X^{**}}^{-1}\otimes \id_{X^*}) = (\ev_{X^*}\otimes \id_X)\circ(\id_{X^{**}}\otimes c_{X,X^*})$. Therefore, using the properties of the duality, we obtain $\Tr^{L,R}_{X\otimes X}(c_{X,X}^{-1})=\ev_X\circ c_{X,X^*} \circ \coev_X$, which leads to the conclusion that $\dim^R(X) = \theta_X\Tr^{L,R}_{X\otimes X}(c_{X,X}^{-1})$.

But the Lemma \ref{lem:trace-mixed-dual} shows that $\Tr_{X^*\otimes X^*}^{L,R}((c_{X,X}^{-1})^*) = \Tr_{X\otimes X}^{L,R}(c_{X,X}^{-1})$ and as $(c_{X,X})^* = c_{X^*,X^*}$ (\emph{c.f.} \cite[Exercise 8.9.2]{egno}) we finally obtain
\[
  \theta_X^{-1}\dim^R(X) = \Tr^{L,R}_{X\otimes X}(c_{X,X}^{-1})= \Tr^{L,R}_{X^*\otimes X^*}(c_{X^*,X^*}^{-1})= \theta_{X^*}^{-1}\dim^R(X^*),
\]
as expected.
\end{proof}

Therefore, for a simple object $X$, its left and right quantum dimensions are equal if and only if $\theta_{X^*}=\theta_X$. We then recover the following well known result \cite[Appendix A.4]{henriques-penneys-tener}.

\begin{corollary}
  \label{cor:spherical-ruban}
  Let $\mathcal{C}$ be a braided pivotal fusion category. The pivotal structure is spherical if and only if the associated twist $\theta$ is a ribbon, that is satisfies $\theta_{X^*}=(\theta_X)^*$ for any object $X$.
\end{corollary}

\begin{proof}
  The category being semisimple, for an object $X=\bigoplus_{Z \in \Irr(\mathcal{C})}Z^{\oplus n_Z}$, we have $\dim^{R}(X) = \sum_{Z\in\Irr(\mathcal{C})} n_Z\dim^{R}(Z)$, $\dim^{L}(X) = \sum_{Z\in\Irr(\mathcal{C})} n_Z\dim^{L}(Z)$, $\theta_{X^*} = \bigoplus_{Z\in\Irr(\mathcal{C})}(\theta_{Z^*}\id_{Z^*})^{\oplus n_Z}$ and $(\theta_X)^* = \bigoplus_{Z\in\Irr(\mathcal{C})}(\theta_Z\id_{Z^*})^{\oplus n_Z}$. The result follows immediately from the Proposition \ref{prop:twist-dim}.
\end{proof}

\begin{definition}
  Let $\mathcal{C}$ be a braided pivotal fusion category equipped with a twist $\tilde{\theta}$. The \emph{Gauss sums} of the category $\mathcal{C}$ are
  \[
  \tau^{\pm}(\mathcal{C},\tilde{\theta}) := \sum_{X\in\Irr(\mathcal{C})}\tilde{\theta}_X^{\pm}\lvert X \rvert^2.
\]
  If the twist $\tilde{\theta}$ is the one obtained from the pivotal structure, we simply denote these sums by $\tau^{\pm}(\mathcal{C})$.
\end{definition}

The relation $\tilde{\theta}_{X\otimes Y} = (\tilde{\theta}_X\otimes\tilde{\theta}_Y) \circ c_{Y,X}\circ c_{X,Y}$ gives, by taking the right quantum trace,
\begin{equation}
  \tilde{\theta}_X\tilde{\theta}_YS_{X,Y}^{R,R} = \sum_{Z\in\Irr(\mathcal{C})}N_{X,Y}^Z\dim^R(Z)\tilde{\theta}_Z,\label{eq:twist-S-matrice}
\end{equation}
for $X$ and $Y$ simple objects of $\mathcal{C}$.

\begin{lemma}
  \label{lem:twist-tau+}
  Let $Y$ be a simple object of a braided pivotal fusion category $\mathcal{C}$ equipped with a twist $\tilde{\theta}$. Then
  \begin{equation}
  \sum_{X\in\Irr(\mathcal{C})}\tilde{\theta}_X\dim^{L}(X)S^{R,R}_{X,Y} = \tilde{\theta}_Y^{-1}\dim^R(Y)\tau^+(\mathcal{C},\tilde{\theta}).\label{eq:twist-tau+}
  \end{equation}
\end{lemma}

\begin{proof}
  The proof is essentially the same as the one of \cite[Lemma 8.15.2]{egno}. Using \eqref{eq:twist-S-matrice}, we have
  \begin{align*}
    \sum_{X\in\Irr(\mathcal{C})}\tilde{\theta}_X\dim^{L}(X)S^{R,R}_{X,Y}
    &= \tilde{\theta}_Y^{-1}\sum_{X,Z\in\Irr(\mathcal{C})}N_{X,Y}^Z\dim^R(Z)\dim^L(X)\tilde{\theta}_Z\\
    &= \tilde{\theta}_Y^{-1}\sum_{Z\in\Irr(\mathcal{C})}\dim^R(Z)\tilde{\theta}_Z\sum_{X\in\Irr(\mathcal{C})}N_{Z^*,Y}^{X^*}\dim^R(X^*)\\
    &= \tilde{\theta}_Y^{-1}\dim^R(Y)\sum_{Z\in\Irr(\mathcal{C})}\tilde{\theta}_Z\lvert Z\rvert^2.
  \end{align*}
\end{proof}

We have a similar formula for $\theta^{-1}$ using simultaneously the Propositions \ref{prop:rev-S-mat} and \ref{prop:twist-dim}.

\begin{lemma}
  \label{lem:twist+tau-}
  Let $Y$ be a simple object of a braided pivotal fusion category $\mathcal{C}$ equipped with the twist $\theta$ associated to the pivotal structure. Then
  \begin{equation}
    \sum_{X\in\Irr(\mathcal{C})}\theta^{-1}_X\dim^R(X)S^{R,R}_{X,Y} = \theta_Y\dim^R(Y)\tau^-(\mathcal{C}).\label{eq:twist+tau-}
  \end{equation}
\end{lemma}

\begin{proof}
  Using the fact that $\theta^{\rev} = \theta^{-1}$ is a twist for the category $\mathcal{C}^{\rev}$, we deduce from Lemma \ref{lem:twist-tau+} that
  \[
    \sum_{X\in\Irr(\mathcal{C})}\theta^{\rev}_X\dim^L(X)S^{\rev,R,R}_{X,Y} = (\theta^{\rev}_Y)^{-1}\dim^R(Y)\tau^+(\mathcal{C}^{\rev},\theta^{\rev}).
  \]
  From the proposition \ref{prop:rev-S-mat}, one has $\dim^L(X)S^{\rev,R,R}_{X,Y}=\dim^R(X)S^{R,R}_{X^*,Y}$ and as $\mathcal{C}$ and $\mathcal{C}^{\rev}$ have the same simple objects, we have $\tau^+(\mathcal{C}^{\rev},\theta^{\rev}) = \tau^-(\mathcal{C})$. Therefore
  \begin{align*}
    \theta_Y\dim^R(Y)\tau^-(\mathcal{C})
    &= \sum_{X\in\Irr(\mathcal{C})}\theta^{-1}_X\dim^L(X)S^{\rev,R,R}_{X,Y}\\
    &= \sum_{X\in\Irr(\mathcal{C})}\theta^{-1}_X\dim^R(X)S^{R,R}_{X^*,Y}\\
    &= \sum_{X\in\Irr(\mathcal{C})}\theta^{-1}_{X^*}\dim^L(X)S^{R,R}_{X^*,Y},
  \end{align*}
the last equality being the Proposition \ref{prop:twist-dim}. As $X\mapsto X^*$ is a bijection of $\Irr(\mathcal{C})$, we conclude using $\dim^L(X)=\dim^R(X^*)$.
\end{proof}

\subsection{Nondegenerate braided pivotal categories}
\label{sec:nondeg}

It is well known that a modular category gives rise to a projective representation of $SL_2(\mathbb{Z})$. We aim to generalize this result to categories with a pivotal structure which is not necessarily spherical.

\boitegrise{{\bf Hypothesis.} {\emph{In this section, we suppose that the category $\mathcal{C}$ is a nondegenerate braided pivotal fusion category.}}}{0.8\textwidth}

All ring homomorphisms $\Gr(\mathcal{C})\rightarrow \Bbbk$ are then of
the form $s_{X}^R$ for $X$ a simple object of $\mathcal{C}$. Indeed,
such a ring homomorphism is the same as a linear character of the
$\Bbbk$-algebra $\Bbbk\otimes_{\mathbb{Z}}\Gr(\mathcal{C})$, which is
semi-simple: if
$e_X=(\sum_{Y\in\Irr(\mathcal{C})}s^R_X(Y)s^R_X(Y^*))^{-1}\sum_{Y\in\Irr(\mathcal{C})}s^R_X(Y^*)[Y]$
is the idempotent associated to $s^R_X$ then
$\Bbbk\otimes_{\mathbb{Z}}\Gr(\mathcal{C})=\bigoplus_{X\in\Irr(\mathcal{C})}\Bbbk
e_X$ since $s_X^R=s_Y^R$ if and only if $X$ and $Y$ are isomorphic
($\sum_{Y\in\Irr(\mathcal{C})}s^R_X(Y)s^R_X(Y^*)$ is non-zero by
\cite[Lemma 8.20.9]{egno}). The map $Y\mapsto s_{X}^R(Y^*)$ is a character of $\Gr(\mathcal{C})$ hence equal to $s_{\bar{X}}^R$ for some $\bar{X}\in\Irr(\mathcal{C})$. This defines an involution $\bar{~}$ on $\Irr(\mathcal{C})$. Note that if the pivotal structure is spherical, this involution is nothing more than the duality. 

\begin{proposition}
  \label{prop:1barinv}
  The object $\bar{\1}$ is invertible and $\bar{X}\simeq X^*\otimes \bar{\1}$.
\end{proposition}

\begin{proof}
Let $\mathcal{C}_{\mathrm{ad}}$ be the adjoint subcategory of $\mathcal{C}$ \cite[Definition 4.14.5]{egno}: it is the smallest tensor Serre subcategory of $\mathcal{C}$ containing all $X\otimes X^*$ for $X$ simple. For $\mathcal{K}$ a subcategory of $\mathcal{C}$ we denote by $\mathcal{K}^{\mathrm{co}}$ the commutator of $\mathcal{K}$ \cite[Definition 4.14.10]{egno}: it is the Serre tensor subcategory generated by all simple objects $X$ of $\mathcal{C}$ such that $X\otimes X^*\in\mathcal{K}$. Applying \cite[Proposition 8.22.6]{egno} to $\mathcal{K}=\mathcal{C}$, we obtain that $\mathcal{Z}_{\mathcal{C}}(\mathcal{C}_{\mathrm{ad}}) = (\mathcal{Z}_{\mathrm{sym}}(\mathcal{C}))^{\mathrm{co}}$, where $\mathcal{Z}_{\mathcal{C}}(\mathcal{C}_{\mathrm{ad}})$ is the centralizer of $\mathcal{C}_{\mathrm{ad}}$ in $\mathcal{C}$ \cite[Definition 8.20.1]{egno}. As $\mathcal{Z}_{\mathrm{sym}}(\mathcal{C})$ is generated by $\1$, we obtain that $(\mathcal{Z}_{\mathrm{sym}}(\mathcal{C}))^{\mathrm{co}}$ is the pointed subcategory of $\mathcal{C}$.

Now, let us show that $\bar{\1}$ is indeed in $\mathcal{Z}_{\mathcal{C}}(\mathcal{C}_{\mathrm{ad}})$. By \cite[Proposition 8.20.5]{egno}, it suffices to show that $S^{R,R}_{\bar{\1},Y}=\dim^R(\bar{\1})\dim^R(Y)$ for all simple objects $Y$ in $\mathcal{C}_{\mathrm{ad}}$. By definition of $\bar{\1}$, we have $S^{R,R}_{\bar{\1},Y}=\dim^R(\bar{\1})\dim^L(Y)$. Now, we consider the slope, which is an tensor autofunctor of $\id_{\mathcal{C}}$ \cite[Section 4.C]{bruguieres-virelizier}. On a simple object $X$ it is defined as:
\[
  \mathrm{Sl}_X := \frac{\dim^L(X)}{\dim^R(X)}\id_X.
\]
It is now immediate that for all $Y$ in $\mathcal{C}_{\mathrm{ad}}$ we have $\mathrm{Sl}_Y=\id_Y$ since $\dim^R(X^*)=\dim^L(X)$ for any simple object $X$. This shows that for any simple object $Y$ in $\mathcal{C}_{\mathrm{ad}}$ the left and right quantum dimensions coincide and therefore $S^{R,R}_{\bar{\1},Y}=\dim^R(\bar{\1})\dim^R(Y)$. Hence, we can conclude that $\bar{\1}$ is an invertible object.

  Now, as $\bar{\1}$ is invertible, $X\otimes \bar{\1}$ is simple for any simple object $X$. Showing that $s_{X^*\otimes \bar{\1}}^R(Y) = s_{X}^R(Y^*)$ for any simple object $Y$ ends the proof:
  \begin{align*}
    s^R_{X^*\otimes\bar{\1}}(Y)=\frac{\dim^R(Y)}{\dim^R(X^*\otimes \bar{\1})}s^R_{Y}(X^*\otimes\bar{\1})&=\frac{\dim^R(Y)}{\dim^R(X^*\otimes \bar{\1})}s^R_{Y}(X^*)s_Y^R(\bar{\1})\\
    &=\frac{S^{R,R}_{Y,X^*}}{\dim^R(X^*)\dim^R(Y)}s^R_{\bar{\1}}(Y)\\
    &=\frac{S^{L,L}_{Y^*,X}}{\dim^R(X^*)\dim^R(Y)}\dim^R(Y^*)\\
    &=\frac{S^{R,R}_{Y^*,X}}{\dim^R(X)}\\
    &=s^R_X(Y^*).
  \end{align*}
\end{proof}

\begin{corollary}
\label{cor:symmetry-bar}
 Under the same hypothesis, for any simple objects $X$ and $Y$ we have $S^{R,R}_{\bar{X},Y}=S^{R,R}_{X,\bar{Y}}$.
\end{corollary}

\begin{proof}
  By definition of $s_{\bar{X}}^R$, we have $S^{R,R}_{\bar{X},Y}=\dim^R(\bar{X})s_{\bar{X}}^R(Y)$. But $\dim^R(\bar{X})=\dim^R(\bar{\1})\dim^R(X^*)$ and $\dim^R=s_{\1}^R$ so that
  \[
    \dim^R(\bar{\1})\dim^R(X^*) = \dim^R(\bar{\1})s_{\bar{\1}}^R(X)=S^{R,R}_{\bar{\1},X} = \dim^R(X)s_X^R(\bar{\1}).
  \]
Hence we have $S^{R,R}_{\bar{X},Y}=\dim^R(X)s_{X}^R(Y^*\otimes \bar{\1})$ which leads to the conclusion.
\end{proof}

If the pivotal structure is spherical, the square of the $S$-matrix is well known: up to a scalar multiple, it is the permutation matrix given by the duality on simple objects (see \cite[8.14]{egno} for further details). Let $E$ be the square matrix such that $E_{X,Y}=\delta_{X,\bar{Y}}$.

\begin{proposition}
  \label{prop:square_S}
  Let $\mathcal{C}$ be a nondegenerate braided pivotal fusion category. Then $(S^{R,R})^2 = \dim(\mathcal{C})\dim^R(\bar{\1})E$.
\end{proposition}

\begin{proof}
  Since $\mathcal{C}$ is nondegenerate, for $X,Y\in\Irr(\mathcal{C})$, the equality $s_X^R=s_Y^R$ as characters of $\Gr(\mathcal{C})$ holds if and only if $X=Y$.

  Suppose $Y\neq \bar{Z}$. We have, thanks to the orthogonality of characters \cite[Lemma 8.14.1]{egno},
\[
\sum_{X\in\Irr(C)}S^{R,R}_{Y,X}S^{R,R}_{X,Z}=\dim^{R}(Y)\dim^{R}(Z)\sum_{X\in\Irr(\mathcal{C})}s_{Y}^R(X)s_{\bar{Z}}^R(X^*)=0.
\]
It remains to compute $(S^{R,R})^2_{Y,\bar{Y}}$:
\begin{align*}
  \sum_{X\in\Irr(C)}S^{R,R}_{Y,X}S^{R,R}_{X,\bar{Y}}&=\sum_{X,W\in\Irr(\mathcal{C})}N_{Y,\bar{Y}}^W\dim^R(X)S^{R,R}_{X,W}\\
                                                  &=\sum_{W\in\Irr(\mathcal{C})}\dim^{R}(W)N_{Y,\bar{Y}}^{W}\sum_{X\in\Irr(\mathcal{C})}\dim^{R}(X)s_{W}^R(X).
\end{align*}
As $\dim^{R}(X) = s_{\1}^R(X)=s_{\bar{\1}}^R(X^*)$, the second sum is zero unless $W=\bar{\1}$ and is equal to
\[
  \sum_{X\in\Irr(\mathcal{C})}\dim^{R}(X)s_{\bar{\1}}^R(X)=\sum_{X\in\Irr(\mathcal{C})}\dim^{R}(X)\dim^{L}(X)=\dim(\mathcal{C}).
\]
Moreover, as $\bar{Y}\simeq Y^*\otimes \bar{\1}$, we have
$N_{Y,\bar{Y}}^{\bar{\1}}=1$ and $(S^{R,R})^2_{Y,\bar{Y}}=\dim^{R}(\bar{\1})\dim(\mathcal{C})$.
\end{proof}

\begin{corollary}[Verlinde formula]
  Let $\mathcal{C}$ be a nondegenerate braided pivotal fusion category and $X,Y,Z \in\Irr(\mathcal{C})$. The structure constants of
  $\Gr(\mathcal{C})$ are given by 
\[
  N_{X,Y}^Z=\frac{1}{\dim(\mathcal{C})\dim^R(\bar{\1})}\sum_{W\in\Irr(\mathcal{C})}\frac{S^{R,R}_{W,X}S^{R,R}_{W,Y}S^{R,R}_{W,\bar{Z}}}{\dim^R(W)}.
\]
\end{corollary}

Recall that giving a pivotal structure on a braided monoidal category is equivalent to endowing the category with a twist using the conventions given before Remark \ref{rk:twist-pivotal}. 

\begin{lemma}
  \label{lem:twist-bar}
  Let $\mathcal{C}$ be a nondegenerate braided pivotal fusion
  category. Then for $X$ simple, $\theta_{\bar{X}}=\theta_{\bar{\mathbf{1}}}\theta_X$.
\end{lemma}

\begin{proof}
  Taking the right quantum trace of the morphism $\theta_{X^*\otimes \bar{\mathbf{1}}}=\theta_{X^*}\otimes\theta_{\bar{\mathbf{1}}}\circ c_{\bar{\mathbf{1}},X^*}\circ c_{X^*,\bar{\mathbf{1}}}$ we obtain
 \begin{align*}
  \theta_{\bar{X}}\dim^R(\bar{X}) = \theta_{X^*}\theta_{\bar{\mathbf{1}}}S^{R,R}_{X^*,\bar{\mathbf{1}}} 
   &= \theta_{X^*}\theta_{\bar{\mathbf{1}}}\dim^R(\bar{\mathbf{1}})s^R_{\bar{\mathbf{1}}}(X^*)\\
   &= \theta_{X^*}\theta_{\bar{\mathbf{1}}} \dim^R(\bar{\mathbf{1}}) \dim^R(X).
  \end{align*}
  The equality $\theta_{\bar{X}}=\theta_{\bar{\mathbf{1}}}\theta_X$ follows then immediately from the fact that $\bar{X}\simeq X^*\otimes \bar{\mathbf{1}}$ and from the Proposition \ref{prop:twist-dim}.
\end{proof}

\begin{remark}
  Taking for $X$ the simple object $\bar{\1}$, we find that $\theta_{\bar{\1}}^2=1$.
\end{remark}

As in the case of a spherical category \cite[Proposition 8.15.4]{egno}, the Gauss sums satisfy $\tau^+(\mathcal{C})\tau^-(\mathcal{C})=\dim(\mathcal{C})$ and hence are non-zero.

\begin{proposition}
  \label{prop:twist1bar}
  Let $\mathcal{C}$ be a nondegenerate braided pivotal fusion category. Then $\theta_{\bar{\1}}=1$, where $\theta$ is the twist associated to the pivotal structure. Therefore for any simple object $X$, one has $\theta_{\bar{X}}=\theta_X$.
\end{proposition}

\begin{proof}
  Using the fact that $\mathcal{C}$ is nondegenerate, we have, as in \cite[Corollary 8.15.5]{egno}
\[
  \sum_{X\in \Irr(\mathcal{C})}\theta_X^{-1}\dim^{R}(X)S^{R,R}_{X,Y} = \theta_{\bar{Y}}\dim^{L}(\bar{Y})\tau^-(\mathcal{C})\dim^{R}(\bar{\1}).
\]
As $\bar{Y} \simeq Y^*\otimes \bar{\1}$, $\theta_{\bar{Y}}=\theta_{\bar{\1}}\theta_{Y}$ and $\lvert \1 \rvert^2 = 1$ we have
\[
\sum_{X\in \Irr(\mathcal{C})}\theta_X^{-1}\dim^{R}(X)S^{R,R}_{X,Y} = \theta_{\bar{\1}}\theta_{Y}\dim^{R}(Y)\tau^-(\mathcal{C}).
\]
This equality for $Y=\1$, together with Lemma \ref{lem:twist+tau-}, show that $\theta_{\bar{\1}}=1$.
\end{proof}

The group $SL_2(\mathbb{Z})$ has a presentation given by:
\[
  \langle \mathfrak{s},\mathfrak{t}\mid \mathfrak{s}^4=1,(\mathfrak{s}\mathfrak{t})^3=\mathfrak{s}^2\rangle
\]
by choosing 
\[
  \mathfrak{s} =
  \begin{pmatrix}
    0&-1\\
    1&0
  \end{pmatrix}
  \qquad \text{and} \qquad
  \mathfrak{t} =
  \begin{pmatrix}
    1&1\\
    0&1
  \end{pmatrix}.
\]

We now choose an embedding $\Bbbk_{\mathrm{alg}}\rightarrow\mathbb{C}$. The categorical dimension of $\mathcal{C}$ being a totally positive number \cite[Theorem 7.21.12]{egno}, we denote its positive square root for the chosen embedding by $\sqrt{\dim(\mathcal{C})}$. We moreover choose a square root $\sqrt{\dim^R(\bar{\1})}$ of $\dim^R(\bar{\1})$. The $T$-matrix of $\mathcal{C}$ is the diagonal matrix given by the action of the inverse of the twist $\theta$ on simple objects. We have the non-spherical analogue of \cite[Theorem 8.16.1]{egno}:

\begin{theoreme}\label{thm:sl2rel_nondeg}
  Let $\mathcal{C}$ be a nondegenerate braided pivotal fusion category. We have $(S^{R,R}T)^3=\tau^{-}(\mathcal{C})(S^{R,R})^2$ and $(S^{R,R})^4 = (\dim(\mathcal{C})\dim^R(\bar{\1}))^2\id$. Therefore
\[
  \mathfrak{s}\mapsto \frac{1}{\sqrt{\dim^R(\bar{\1})}\sqrt{\dim(\mathcal{C})}}S^{R,R} \qquad \text{and}\qquad \mathfrak{t}\mapsto T
\]
define a projective representation of $SL_2(\mathbb{Z})$.
\end{theoreme}

We can show that that $S^{R,R}$ and $T$ satisfy $(S^{R,R}T^{-1})^3 =
\tau^+(\mathcal{C})\dim(\mathcal{C})\dim^R(\bar{\1})^2\id$. Indeed, we have
\[
  S^{R,R}TS^{R,R}=\tau^{-}(\mathcal{C})T^{-1}S^{R,R}T^{-1}.
\] 
Multiplying by $S^{R,R}$ on both sides, using Proposition \ref{prop:square_S} and the fact that $\theta_{\bar{X}}=\theta_X$, we have
\[
  \dim(\mathcal{C})^2\dim^R(\bar{\1})^2T = \tau^{-}(\mathcal{C})S^{R,R}T^{-1}S^{R,R}T^{-1}S^{R,R},
\]
which gives $(S^{R,R}T^{-1})^3 =
\tau^+(\mathcal{C})\dim(\mathcal{C})\dim^R(\bar{\1})^2\id$ since $\tau^+(\mathcal{C})\tau^-(\mathcal{C})=\dim(\mathcal{C})$.

Define
$\xi(\mathcal{C}) =
\frac{\tau^+(\mathcal{C})}{\sqrt{\dim(\mathcal{C})}}\sqrt{\dim^R(\bar{\1})}$
so that the images of $\mathfrak{s}$ and $\mathfrak{t}$ satisfy
\[
  \mathfrak{s}^4=\id,\quad
  (\mathfrak{st})^3=\xi(\mathcal{C})^{-1}\mathfrak{s}^2
  \quad \text{and} \quad (\mathfrak{st^{-1}})^3 = \xi(\mathcal{C})\id.
\]

\begin{corollary}
  \label{cor:N-modular}
  Let $\mathcal{C}$ be a nondegenerate braided pivotal fusion category over $\mathbb{C}$. Let 
\[
  \widetilde{S}^{R,R}:=\frac{1}{\sqrt{\dim^R(\bar{\1})}\sqrt{\dim(\mathcal{C})}}S^{R,R}.
\]
Then the finite set $\Irr(\mathcal{C})$, the unit object $\1$ and the matrices $\widetilde{S}^{R,R}$ and $T^{-1}$ define a $\mathbb{N}$-modular datum.
\end{corollary}

\begin{proof}
  It remains to show that the matrix $\widetilde{S}^{R,R}$ is unitary. Since $(\tilde{S}^{R,R})^2$ is the permutation matrix given by the involution $\bar{\phantom{a}}$, we show that
\[
  \overline{\tilde{S}^{R,R}_{X,Y}}=\tilde{S}^{R,R}_{X,\bar{Y}}
\]
for all $X,Y\in\Irr(\mathcal{C})$. Following \cite[Proposition 2.12]{eno}, one shows that for any $X,Y\in \Irr(\mathcal{C})$
\[
   \overline{\left(\frac{\widetilde{S}^{R,R}_{X,Y}}{\widetilde{S}^{R,R}_{1,Y}}\right)} = \frac{\widetilde{S}^{R,R}_{X^*,Y}}{\widetilde{S}^{R,R}_{1,Y}}
\]
and therefore
\[
  \overline{\widetilde{S}^{R,R}_{X,Y}}= \frac{\overline{\widetilde{S}^{R,R}_{1,Y}}}{\widetilde{S}^{R,R}_{1,Y}}\widetilde{S}^{R,R}_{X^*,Y}= \frac{\overline{\widetilde{S}^{R,R}_{1,Y}}}{\widetilde{S}^{R,R}_{1,\bar{Y}}}\widetilde{S}^{R,R}_{X,\bar{Y}},
\]
the last equality following by definition of the involution $\bar{\phantom{a}}$. But 
\[
\overline{\widetilde{S}^{R,R}_{1,Y}} = \frac{\sqrt{\dim^R(\bar{\1})}\dim^R(Y^*)}{\sqrt{\dim(\mathcal{C})}},
\]
since $\sqrt{\dim^R(\bar{\1})}$ is a root of unity and $\overline{\dim^R(Y)}=\dim^R(Y^*)$ (see \cite[Proposition 2.9]{eno}). We conclude using the fact that $\bar{Y}\simeq Y^*\otimes\bar{\1}$.
\end{proof}

We now generalize Vafa's theorem to the setting where the pivotal structure is not necessarily spherical.

\begin{theoreme}
  \label{thm:vafa-non-dg}
  Let $\mathcal{C}$ be a nondegenerate braided pivotal fusion category. Then for any simple object $X$, the twist $\theta_X$ is a root of unity and so is $\xi(\mathcal{C})$.
\end{theoreme}

\begin{proof}
  We follow the proof of Vafa \cite{vafa}. For $X_1,X_2$ and $X_3$ objects in $\mathcal{C}$, we define the following endomorphisms of $X_1\otimes X_2\otimes X_3$:
  \begin{align*}
    \theta_1&:=\theta_{X_1}\otimes \id_{X_2\otimes X_3},& \theta_2&:=\id_{X_1}\otimes \theta_{X_2}\otimes \id_{X_3}, & \theta_3&:=\id_{X_1\otimes X_2}\otimes \theta_{X_3},\\
    \theta_{12}&:=\theta_{X_1\otimes X_2}\otimes \id_{X_3}, & \theta_{23}&:=\id_{X_1}\otimes\theta_{X_2\otimes X_3}, & \theta_{13} &:= (\id_{X_1}\otimes c_{X_3,X_2})\circ (\theta_{X_1\otimes X_3}\otimes\id_{X_2}) \circ (\id_{X_1}\otimes c_{X_3,X_2}^{-1}).
  \end{align*}

It is not difficult to show that
\begin{equation}
  \theta_{12}\theta_{13}\theta_{23} = \theta_{X_1\otimes X_2\otimes X_3}\theta_1\theta_2\theta_3.
  \label{eq:lantern}
\end{equation}

We now fix a simple object $X$ and consider this identity for $X_1=X_3=X$ and $X_2=X^*$. The vector space $V_X:=\Hom_{\mathcal{C}}(X,X\otimes X^* \otimes X)$ is non-zero, and therefore the identity \eqref{eq:lantern} gives rise to an identity of operators in this vector space. We compute the determinant of both sides. The determinant of the right-hand side is
\[
  \theta_X^{3\dim(V_X)}\theta_{X^*}^{\dim(V_X)}.
\]
Similarly to \cite{bakalov-kirillov}, one can show that the determinant of the left-hand side is
\[
  \prod_{Y\in \Irr(\mathcal{C})}\theta_Y^{A_{X,Y}},
\] 
where $A_{X,Y} = 2N_{X,X^*}^YN_{Y,X}^X+N_{X,X}^YN_{Y,X^*}^X$. 

Since $\theta_{X^*}=\theta_X \frac{\dim^L(X)}{\dim^R(X)} = \theta_X \frac{s_X^R(\bar{\1})}{\dim^R(\bar{\1})}$, we obtain:
\[
  \prod_{Y\in\Irr(\mathcal{C})} \theta_Y^{M_{X,Y}} = \frac{s_X^R(\bar{\1})}{\dim^R(\bar{\1})}, 
\]
where $M=A-4\id$. As $\dim(V_X) = \frac{1}{3}\sum_{Y\in\Irr(\mathcal{C})}A_{X,Y}$, the matrix $M$ is strictly diagonally dominant and therefore invertible. The object $\bar{\1}$ being invertible, $\frac{s_X^R(\bar{\1})}{\dim^R(\bar{\1})}$ is a root of unity, which ends the proof because we can diagonalize $M$ by row and column operations.

The assertion on $\xi(\mathcal{C})$ follows now immediately from Theorem \ref{thm:sl2rel_nondeg}.
\end{proof}

\begin{exemple}
  We end this section with the example of the $\Bbbk$-linear semisimple category $\Ve_{\mathbb{Z}/n\mathbb{Z}}$ of $\mathbb{Z}/n\mathbb{Z}$-graded vector spaces, with $n$ odd. We denote the degree of an homogeneous element $x$ in a $\mathbb{Z}/n\mathbb{Z}$-graded vector space $X$ by $\lvert x \rvert$. For $k\in\mathbb{Z}/n\mathbb{Z}$, we denote by $\delta_{k}$ the one dimensional vector space in degree $k$. Therefore every simple object in $\Ve_{\mathbb{Z}/n\mathbb{Z}}$ is isomorphic to some $\delta_k$.

The classification of braided monoidal structures on $\Ve_{\mathbb{Z}/n\mathbb{Z}}$ is given in \cite[Section 3]{joyal-street}. Given $\zeta$ an element in $\Bbbk$ such that $\zeta^n=1$, the associativity constraint is
\[
  \alpha_{X,Y,Z}\colon (x\otimes y)\otimes z \mapsto x\otimes(y\otimes z),
\]
and the braiding is
\[
  c_{X,Y}^{\zeta}\colon x\otimes y \mapsto \zeta^{\lvert x\rvert \lvert y \rvert}y\otimes x 
\]
for $x$ and $y$ homogeneous elements. The simple object $\delta_k$ is in the symmetric center of $\Ve_{\mathbb{Z}/n\mathbb{Z}}$ equipped with the braiding $c^{\zeta}$ if and only if we have $\zeta^{2k}=1$. Hence the category is nondegenerate if and only if $\zeta$ is a primitive $n$-th root of unity. From now on, we will suppose that it is indeed the case.

The usual spherical structure on vector spaces is still a spherical structure on $\mathbb{Z}/n\mathbb{Z}$-graded vector spaces. But we can define other pivotal structures for $\xi\in\Bbbk$ such that $\xi^n=1$:
\[
  a_{X}^{\xi}\colon x \mapsto (\varphi\in X^* \mapsto \xi^{\lvert x \rvert}\varphi(x)),
\]
for $x$ homogeneous. Fix such a $\xi$, and with the pivotal structure $a^{\xi}$, we have:
\[
  \dim^R(\delta_k) = \xi^k \quad \text{and} \quad \dim^L(\delta_k) = \xi^{-k}.
\]

  The $S$-matrix and the twist are given by:
\[
  S^{R,R}_{\delta_k,\delta_l} = \xi^{k+l}\zeta^{2kl}\quad\text{and}\quad \theta_{\delta_k}=\xi^k\zeta^{k^2}.
\]
Finally, the object $\bar{\1}$ is $\delta_{-k_0}$ where $k_0$ is such that $\xi=\zeta^{k_0}$.
\end{exemple}


\section{Slightly degenerate fusion category}
\label{sec:degfus}

The main object of study of this section is slightly degenerate
fusion categories. These are braided fusion categories with symmetric
center equal to superspaces. We give an analogue of the Verlinde
formula, where the structure constants involved are the ones of a quotient
of the Grothendieck ring of $\mathcal{C}$; these structure
constants can be negative.

\boitegrise{{\bf Hypothesis.} {\emph{In this section, we assume that $\mathcal{C}$ is a slightly degenerate braided pivotal fusion category.}}}{0.8\textwidth} 

Denote by $\e$ the invertible object generating the symmetric center of
$\mathcal{C}$. As $\mathbb{Z}_{\mathrm{sym}}(\mathcal{C})= \sVect$, the twist of
$\e$ is either $1$ or $-1$. In the first case, $\e$
is of quantum dimension $-1$ whereas in the second case, it is of
quantum dimension $1$. 

Tensoring by $\e$ gives an involution on the set of isomorphism classes of simple objects. According to \cite[Proposition 9.15.4]{egno}, this involution has no fixed points. We choose
$J\subseteq\Irr(\mathcal{C})$ a set of representatives of orbits
of this involution such that $\1\in J$. The $S$-matrix of
$\mathcal{C}$ is then of rank half its size by \cite[Theorem
8.20.7]{egno} and $s_X^R=s_Y^R$ if and only if $X\simeq Y$ or $X\simeq
Y\otimes \e$ by \cite[Lemma 8.20.8]{egno}.

\begin{lemma}
  The $S$-matrix gives the characters of the quotient ring
  ${\Gr(\mathcal{C})/([\e]-\dim(\e)[\1])}$.
\end{lemma}

\begin{proof}
  Denote by $A$ the ring
  $\Gr(\mathcal{C})/([\e]-\dim(\e)[\1])$. It
  has a $\mathbb{Z}$-basis given by the elements of $J$ and for
  $X,Y,Z \in J$, the structure constants in $A$ are given by
  \[
    sN_{X,Y}^Z:=N_{X,Y}^Z+\dim(\e)N_{X,Y}^{Z\otimes \e}.
  \]
  As $\e$ is in the symmetric center,
  $S^{R,R}_{X,Y\otimes\e}=\dim(\e)S^{R,R}_{X,Y}$. Therefore, for
  $W,X,Y \in J$
  \begin{align*}
    s_{W}^R(X\otimes Y) &=
                          \sum_{Z\in\Irr(\mathcal{C})}N_{X,Y}^Zs_{W}^R(Z)\\ 
    &= \sum_{Z\in J}(N_{X,Y}^Zs_W^R(Z)+N_{X,Y}^Zs_W^R(Z\otimes
      \e))\\
    &=\sum_{Z\in J}sN_{X,Y}^Zs_W^R(Z),
  \end{align*}
  and $s_W^R$ is indeed a character of $A$.
\end{proof}

As in Section \ref{sec:nondeg}, any ring homomorphism $\Gr(\mathcal{C})/([\e]-\dim(\e)[\1])\rightarrow\Bbbk$ is of the form $s_X^R$ for a unique $X\in J$ and we define an involution $\bar{~}$ on
$J$: for any $X\in J$, there exists a unique
$\bar{X}\in J$ such that for all $Y\in \Irr(\mathcal{C})$
\[
  s_{X}^R(Y^*)=s_{\bar{X}}^R(Y).
\]
Again, if the pivotal structure is spherical, then $\bar{X}\simeq X^*$ or
$\bar{X}\simeq X^*\otimes\e$ whether $X^*$ is in $J$ or not.

For a slightly degenerate category, we define its superdimension by
\[
  \sdim(\mathcal{C}):=\sum_{X\in J}|X|^2=\frac{1}{2}\dim(\mathcal{C}).
\]
Note that this does not depend on the choice of $J$.

\begin{proposition}
  Let $\mathcal{C}$ be a slightly degenerate braided pivotal fusion category. Then the simple object $\bar{\1}$ is invertible and for $X\in J$ we have $\bar{X}\simeq X^*\otimes\bar{\1}$ or $\bar{X}\simeq X^*\otimes\bar{\1}\otimes\e$.
\end{proposition}

\begin{proof}
  As in the proof of Proposition \ref{prop:1barinv}, we have $(\mathcal{Z}_{\mathrm{sym}}(\mathcal{C}))^{\mathrm{co}}=\mathcal{C}_{\mathrm{pt}}$: for any simple object $X$ we have $X\not\simeq \e\otimes X$, therefore a simple object $X$ is in $(\mathcal{Z}_{\mathrm{sym}}(\mathcal{C}))^{\mathrm{co}}$ if and only if $X$ is invertible. We conclude that $\bar{\1}$ is invertible similarly to the nondegenerate case.

  As in the proof of Proposition \ref{prop:1barinv}, a simple calculation shows that $s_{X^*\otimes \bar{\1}}^R(Y) = s_{X}^R(Y^*)$ for any simple object $Y$. Therefore $\bar{X}\simeq X^*\otimes\bar{\1}$ if $X^*\otimes\bar{\1}\in J$ and $\bar{X}\simeq X^*\otimes\bar{\1}\otimes\e$ otherwise.
\end{proof}

Let $E$ be the square matrix such that $E_{X,Y}=(\dim(\e))^{\delta_{X^*\otimes\bar{\1}\not\in J}}\delta_{X,\bar{Y}}$ for $X$ and $Y$ in $J$.

\begin{proposition}
  Let $\mathcal{C}$ be a slightly degenerate pivotal braided fusion category. The matrix $\mathbf{S}=(S_{X,Y}^{R,R})_{X,Y\in J}$ satisfies $\mathbf{S}^2=\sdim(\mathcal{C})\dim^R(\bar{\1})E$.
\end{proposition}

\begin{proof}
  First, note that if $\chi_1$ and $\chi_2$ are two different characters of
  $\Gr(\mathcal{C})$ such that $\chi_1(\e)=\chi_2(\e)$ then
\[
  \sum_{W\in J}\chi_1(W)\chi_2(W^*)=0.
\]
Indeed, by the usual orthogonality of characters,
\[
  0 = \sum_{W\in J}\chi_1(W)\chi_2(W^*)+\sum_{W\in J}\chi_1(W\otimes\e)\chi_2(W^*\otimes\e) = 2\sum_{W\in J}\chi_1(W)\chi_2(W^*).
\]

  Using the fact that for $X,Y \in J$, $s^R_X=s^R_Y$ if and only if
  $X=Y$ we show as in the nondegenerate case that
  $(\mathbf{S}^2)_{X,Y}=0$ and that
  $(\mathbf{S}^2)_{X,\bar{X}}=\sdim(\mathcal{C})\dim^R(\bar{\1})sN_{X,\bar{X}}^{\bar{\1}}$. It
  is then easy to see that
  $sN_{X,\bar{X}}^{\bar{\1}}=1$ if
  $X^*\otimes\bar{\1}\in J$ and
  $sN_{X,\bar{X}}^{\bar{\1}}=\dim(\e)$ if
  $X^*\otimes\bar{\1}\not\in J$.
\end{proof}

\begin{corollary}[Verlinde formula]
  Let $\mathcal{C}$ be a slightly degenerate braided pivotal fusion category and $X,Y,Z\in J$. The structure constants of $\Gr(\mathcal{C})/([\1]-\dim(\e)[\e])$ are given by
\[
  sN_{X,Y}^Z=
  \frac{(\dim(\e))^{\delta_{Z^*\otimes \1\not\in J}}}{\sdim(\mathcal{C})\dim^R(\bar{\1})}\sum_{W\in J}\frac{\mathbf{S}_{W,X}\mathbf{S}_{W,Y}\mathbf{S}_{W,\bar{Z}}}{\dim^R(W)}.
\]
\end{corollary}

We now study the $T$-matrix of a slightly degenerate braided pivotal
fusion category. The equality
$\theta_{X^*}\dim^R(X)=\theta_X\dim^L(X)$ has been proven without
assumption on the degeneracy of the category $\mathcal{C}$. As in the
nondegenerate setting, it is easy to prove that
$\theta_{\bar{X}}=\theta_{X}\theta_{\bar{\1}}$ if $X^*\otimes \bar{\1}\in J$ and $\theta_{\bar{X}}=\theta_{X}\theta_{\bar{\1}}\theta_{\e}$ otherwise. Moreover, $\theta_{\bar{\1}}^2=1$. For $X$ and $Y\in J$, we have the relation
\begin{align*}
  \theta_X\theta_Y\mathbf{S}_{X,Y}&=\sum_{Z\in\Irr(\mathcal{C})}N_{X,Y}^Z\dim^R(Z)\theta_Z\\
                                  &=\sum_{Z\in J}(N_{X,Y}^Z+\dim(\e)\theta_{\e}N_{X,Y}^{Z\otimes\e})\dim^R(Z)\theta_Z.
\end{align*}
But $\dim(\e)\theta_{\e}=-1$ by definition of
$\e$.

\boitegrise{\noindent{\bf Hypothesis.} {\emph{From now on, we suppose that $\dim(\e) = -1$ and $\theta_{\e}=1$: $\sVect$ is equipped with its non-unitary pivotal structure.}}}{0.8\textwidth}

With these assumptions, the structure constants of
$\Gr(\mathcal{C})/([\1]+[\e])$ appear naturally:
\begin{equation}
  \theta_X\theta_Y\mathbf{S}_{X,Y}=\sum_{Z\in J}sN_{X,Y}^Z\dim^R(Z)\theta_Z\label{eq:S_twist}
\end{equation}
for any $X,Y\in J$.

We define the Gauss supersums of the slightly degenerate category $\mathcal{C}$ as
\[
  \s\tau^{\pm}(\mathcal{C}) := \sum_{X\in J}|X|^2\theta_X^{\pm 1} = \frac{1}{2}\tau^{\pm}(\mathcal{C}).
\]
Note that these are independent of the choice of $J$ since $\theta_{\e}=1$.

\begin{proposition}
  The twists and the $S$-matrix satisfy for all $Y\in J$
\[
  \sum_{X\in J}\theta_X\dim^L(X)\mathbf{S}_{X,Y}=\theta_{Y}^{-1}\dim^R(Y)\s\tau^+(\mathcal{C})
\]
and
\[
  \sum_{X\in J}\theta_X^{-1}\dim^R(X)\mathbf{S}_{X,Y}=\theta_{\bar{\1}}\theta_{Y}\dim^R(Y)\s\tau^-(\mathcal{C}).
\]
\end{proposition}

\begin{proof}
  Using \eqref{eq:twist-tau+}, we have
\[
  \sum_{X\in \Irr(\mathcal{C})}\theta_X\dim^L(X)S^{R,R}_{X,Y}=\theta_{Y}^{-1}\dim^R(Y)s\tau^+(\mathcal{C}).
\]
But $\theta_{\varepsilon\otimes X}\dim^L(\varepsilon\otimes X)S^{R,R}_{\varepsilon\otimes X,Y} = \theta_X\dim^L(X)S^{R,R}_{X,Y}$ and then 
\[
  \sum_{X\in \Irr(\mathcal{C})}\theta_X\dim^L(X)S^{R,R}_{X,Y} = 2\sum_{X\in J}\theta_X\dim^L(X)\mathbf{S}_{X,Y}.
\]
We conclude by definition of $s\tau^+(\mathcal{C})$. We do the same for the second assertion, starting from \eqref{eq:twist+tau-}.
\end{proof}

\begin{corollary}
  Let $\mathcal{C}$ be a slightly degenerate braided pivotal fusion category. Then $\theta_{\bar{\1}}=1$.
\end{corollary}

\begin{proof}
  Same as Proposition \ref{prop:twist1bar}.
\end{proof}

We finally conclude this section by giving an analogue of Theorem \ref{thm:sl2rel_nondeg} in the setting of slightly degenerate braided pivotal fusion category. We denote by $\mathbf{T}$ the diagonal matrix with entries $\theta_X^{-1}$ for $X\in J$. We again fix an embedding $\Bbbk_{\mathrm{alg}}\rightarrow\mathbb{C}$ and denote the positive square root of $\sdim(\mathcal{C})$ for this embedding by $\sqrt{\sdim(\mathcal{C})}$. We moreover choose a square root $\sqrt{\dim^R(\bar{\1})}$ of $\dim^R(\bar{\1})$.

\begin{theoreme}\label{thm:sl2rel_sldeg}
  Let $\mathcal{C}$ be a slightly degenerate braided pivotal fusion category. We have $(\mathbf{S}\mathbf{T})^3=\s\tau^{-}(\mathcal{C})\mathbf{S}^2$ and $\mathbf{S}^4 = (\sdim(\mathcal{C})\dim^R(\bar{\1}))^2\id$. Therefore
\[
  \mathfrak{s}\mapsto \frac{1}{\sqrt{\dim^R(\bar{\1})}\sqrt{\sdim(\mathcal{C})}}\mathbf{S} \qquad \text{and}\qquad \mathfrak{t}\mapsto \mathbf{T}
\]
define a projective representation of $SL_2(\mathbb{Z})$.
\end{theoreme}

As in the nondegenerate setting, one may define $\s\xi(\mathcal{C}) = \frac{\s\tau^+(\mathcal{C})}{\sqrt{\sdim(\mathcal{C})}}\dim^R(\bar{\1})$ so that the images of $\mathfrak{s}$ and $\mathfrak{t}$ satisfy:
\[
  \mathfrak{s}^4=\id,\quad (\mathfrak{st})^3=\s\xi(\mathcal{C})^{-1}\mathfrak{s}^2\quad\text{and}\quad (\mathfrak{st}^{-1})^3=\s\xi(\mathcal{C})\id.
\]

\begin{remark}
  If $\dim(\e)=1$ and $\theta_{\e}=-1$, the $S$ and
  $T$-matrices do not necessarily give a representation of
  $SL_2(\mathbb{Z})$.

  Consider the Verlinde modular category $\mathcal{C}(\mathfrak{sl}_2,q)$ where $q$ is a $16$-th root of unity \cite[Section 8.18.2]{egno}. It has $7$ simple objects $V_0=\1,\ldots V_6$. Let $\mathcal{C}$ be the full subcategory of $\mathcal{C}(\mathfrak{sl}_2,q)$ generated by $V_0,V_2,V_4,V_6$. The $S$-matrix and the $T$-matrix of $\mathcal{C}$ are
\[
S=
\begin{pmatrix}
  1 & [3] & [3] & 1\\
  [3] & -1 & -1 & [3]\\
  [3] & -1 & -1 & [3]\\
  1 & [3] & [3] & 1 
\end{pmatrix}
\quad \text{and} \quad
T =
\begin{pmatrix}
1 & 0 & 0 & 0\\
0 & -i& 0 & 0\\
0 & 0 & i & 0\\
0 & 0 & 0 &-1  
\end{pmatrix},
\]
where $[3] = q^{-2}+1+q^2$ and $i=q^4$ is a primitive fourth root of unity. It is immediate that the symmetric center of $\mathcal{C}$ is generated by $V_6$ as a tensor category, and $V_6$ is of dimension $1$ and of twist $-1$. The symmetric center of $\mathcal{C}$ is then equivalent to $\sVect$, the matrices $\mathbf{S}$ and $\mathbf{T}$ are
\[
\mathbf{S}=
\begin{pmatrix}
  1 & [3] \\
  [3] & -1
\end{pmatrix}
\quad \text{and} \quad
\mathbf{T} = 
\begin{pmatrix}
  1 & 0\\
  0 & -i
\end{pmatrix}
\]
and they do not define a projective representation of $SL_2(\mathbb{Z})$.
\end{remark}

\begin{corollary}
  Let $\mathcal{C}$ be a slightly degenerate braided pivotal fusion category over $\mathbb{C}$ with $\dim(\e)=-1$. Let 
\[
  \widetilde{\mathbf{S}}:=\frac{1}{\sqrt{\dim^R(\bar{\1})}\sqrt{\dim(\mathcal{C})}}\mathbf{S}.
\]
Then the finite set $\Irr(\mathcal{C})$, the unit object $\1$ and the matrices $\widetilde{\mathbf{S}}$ and $\mathbf{T}^{-1}$ define a $\mathbb{Z}$-modular datum.
\end{corollary}

\begin{proof}
  It remains to show that the matrix $\widetilde{\mathbf{S}}$ is unitary. Since $\tilde{\mathbf{S}}^2$ is the matrix of a signed permutation, we show that
\[
  \overline{\tilde{\mathbf{S}}_{X,Y}}=(-1)^{\delta_{Y^*\otimes\bar{\1}\not\in J}}\tilde{\mathbf{S}}_{X,\bar{Y}}
\]
for all $X,Y\in\Irr(\mathcal{C})$. As in the proof of Corollary \ref{cor:N-modular}, we show that
\[
  \overline{\tilde{\mathbf{S}}_{X,Y}}=\frac{\dim^R(Y^*\otimes\bar{\1})}{\dim^R(\bar{Y})}\tilde{\mathbf{S}}_{X,\bar{Y}},
\]
and therefore the expected sign appears.
\end{proof}

As in the nondegenerate setting, there is a version of Vafa's theorem, and the proof remains the same.

\begin{theoreme}
  \label{thm:vafa-sl-dg}
  Let $\mathcal{C}$ be a slightly degenerate braided pivotal fusion category. Then for any simple object $X$, the twist $\theta_X$ is a root of unity and so is $\s\xi(\mathcal{C})$.
\end{theoreme}


\section{Braided fusion supercategory associated to a slightly
  degenerate braided fusion category}
\label{sec:supermon}

In this section, we recall the notion of a superfusion category, which
is a fusion category enriched over super vector spaces. We
then construct, following \cite{monoidalsupercat} and \cite{aobc}, a superfusion
category $\hat{\mathcal{C}}$ associated to a slightly degenerate
fusion category $\mathcal{C}$. The braiding and the pivotal structure
of $\mathcal{C}$ naturally endow the supercategory $\hat{\mathcal{C}}$
with a braiding and a pivotal structure. The $S$-matrix of
$\mathcal{C}$ gives then the structure constants of a quotient of the
super Grothendieck ring of $\hat{\mathcal{C}}$.

\subsection{Superfusion categories}
\label{sec:superfus}

By a superspace, we mean a $\mathbb{Z}/2\mathbb{Z}$-graded vector
space. We denote by $|v|$ the parity of an homogeneous element $v$ of a superspace,
by $\sVect$ the category of superspaces of finite dimension with morphisms even linear maps
(\emph{i.e.} mapping even degree to even degree and odd one to odd one). This
category is monoidal with $(V\otimes W)_0=V_0\otimes W_0 \oplus
V_1\otimes W_1$ and $(V\otimes W)_1=V_0\otimes W_1 \oplus
V_1\otimes W_0$, braided with braiding defined by
\[
  c_{V,W}(v\otimes w) = (-1)^{|v||w|}w\otimes v,
\]
for $v$ and $w$ homogeneous elements; this braiding is symmetric.

\begin{remark}
  The category of superspaces with even and odd morphisms is not
  monoidal. This is due to the following relation, known as \emph{superinterchange law}
  \[
    (g\otimes g')\circ (f\otimes f') = (-1)^{|g'||f|}(g\circ f)\otimes
    (g'\circ f'),
  \]
  for $f$ and $g$ homogeneous morphisms. The definition of a monoidal supercategory will naturally satisfy
  this superinterchange law.
\end{remark}

A \emph{supercategory} is a $\sVect$-enriched category; a
\emph{superfunctor} between two supercategories is a $\sVect$-enriched
functor (see \cite[Section 1.2]{enriched} for more details); a \emph{supernatural transformation} $\alpha\colon F\rightarrow G$ between two superfunctors $F$ and $G$ is a collection of morphisms $\alpha_X\colon F(X)\rightarrow G(X)$ satisfying a certain supernaturality condition. A supernatural
transformation is said to be \emph{even} if all its components are even.

Now, given two supercategories $\mathcal{A}$ and $\mathcal{B}$, we
define $\mathcal{A}\boxtimes\mathcal{B}$ as the supercategory whose
objects are pairs $(X,Y)$ with $X\in\mathcal{A}$ and $Y\in\mathcal{B}$
and whose morphisms are given by
$\Hom_{\mathcal{A}\boxtimes\mathcal{B}}((X,X'),(Y,Y'))=\Hom_{\mathcal{A}}(X,Y)\otimes\Hom_{\mathcal{B}}(X',Y')$,
the tensor product being the one of superspaces. The composition in $\mathcal{A}\boxtimes\mathcal{B}$ is defined using the braiding in
$\sVect$:
\[
  (g\otimes g')\circ (f\otimes f') = (-1)^{|g'||f|}(g\circ f)\otimes
    (g'\circ f').
  \]

  \begin{definition}[{\cite[Definition 1.4]{monoidalsupercat}}]
    A \emph{monoidal supercategory} is a sextuple
    $(\mathcal{C},\otimes,a,\1,l,r)$ where $\mathcal{C}$ is a
    supercategory, $\otimes\colon
    \mathcal{C}\boxtimes\mathcal{C}\rightarrow\mathcal{C}$ is a
    superfunctor, $\1$ is the unit object, and $a\colon (-\otimes
    -)\otimes - \rightarrow -\otimes(-\otimes -)$,
    $l\colon\1\otimes -\rightarrow -$ and $r\colon -\otimes
    \1\rightarrow -$ are even
    supernatural isomorphisms satisfying axioms analogous to the ones
    of a monoidal category.

    A \emph{monoidal superfunctor} between two monoidal
    supercategories $\mathcal{C}$ and $\mathcal{D}$ is a pair $(F,J)$
    where $F\colon \mathcal{C} \rightarrow \mathcal{D}$ is a
    superfunctor and $J\colon F(-)\otimes F(-) \rightarrow F(-\otimes
    -)$ is an even supernatural isomorphism satisfying axioms analogous
    to the ones for a monoidal functor and such that
    $F(\1_{\mathcal{C}})$ is evenly isomorphic to
    $\1_{\mathcal{D}}$.
  \end{definition}

  We now define the notion of \emph{braided monoidal
  supercategory} \cite[Section 2.2]{aobc}. First, let
$\tau\colon\mathcal{C}\boxtimes\mathcal{C}\rightarrow\mathcal{C}$ denote
the superfunctor sending $(X,Y)$ on $Y\otimes X$ and $f\otimes g$ on
$(-1)^{|f||g|}g\otimes f$. A \emph{braiding} on $\mathcal{C}$ is then
an even supernatural isomorphism $c\colon -\otimes -\rightarrow \tau$
satisfying the usual hexagon axioms of a braided monoidal category.

In a monoidal supercategory $\mathcal{C}$, a \emph{left dual} of an object $X\in\mathcal{C}$ is a triple
  $(X^*,\ev_X,\coev_X)$ where $X^*\in\mathcal{C}$, $\ev_X\colon X^*\otimes X \rightarrow \1$
  and $\coev_X\colon \1 \rightarrow X\otimes X^*$ are even
  morphisms satisfying the same axioms of duality in a rigid monoidal category.
Similarly, there is a notion of right dual. A monoidal supercategory
in which any object has a left and right dual is called \emph{rigid}. In a rigid monoidal supercategory, we define the dual $f^*\in\Hom_{\mathcal{C}}(Y^*,X^*)$ of $f\in\Hom_{\mathcal{C}}(X,Y)$ as in any rigid monoidal category. Note that we have $(g\circ f)^* = (-1)^{|f||g|}f^*\circ g^*$ for two homogeneous morphisms $f$ and $g$, the sign being due to the superinterchange law.

Finally, a \emph{superfusion category} is an abelian semisimple rigid monoidal
supercategory with finite number of simple objects, finite dimensional
spaces of morphisms and simple unit object $\1$.

Before giving a definition of a pivotal structure on a rigid monoidal supercategory, let us have a closer look to the case of the supercategory of superspaces. On $\sVect$, we have two pivotal structures. 

The first one is given by the Drinfeld morphism (see \cite[Section 8.9]{egno}). It is indeed a pivotal structure since the braiding is symmetric. This pivotal structure is spherical and the quantum trace is known as supertrace. One can show that for any morphism $f$ between superspaces $V$ and $W$, the diagram
\begin{equation*}
  \begin{tikzcd}
    V \ar{r}{f}\ar{d}[swap]{a_V} & W\ar{d}{a_W}\\
    V^{**} \ar{r}{f^{**}} & W^{**}
  \end{tikzcd}
\end{equation*}
is commutative.

The second one is given by a slight modification of the previous one by a tensor autofunctor of $\id$ given by $v\mapsto (-1)^{\lvert v\rvert}$ on homogeneous elements. With this pivotal structure, one recover the usual trace as the quantum trace. One can show that for any homogeneous morphism $f$ between superspaces $V$ and $W$, the diagram
\begin{equation*}
  \begin{tikzcd}
    V \ar{r}{f}\ar{d}[swap]{a_V} & W\ar{d}{a_W}\\
    V^{**} \ar{r}{(-1)^{\lvert f\rvert}f^{**}} & W^{**}
  \end{tikzcd}
\end{equation*}
is commutative.

For a rigid monoidal supercategory, we define two even superfunctors $\bid$ and $\sbid$. Both are given by $X\mapsto X$ on objects. On morphisms, $\bid$ sends $f$ to $f^{**}$ whereas $\sbid$ sends $f$ to $(-1)^{\lvert f \rvert}f^{**}$. The functor $\sbid$ has been considered in \cite[Section 8.3]{superpivotal}.

A \emph{pivotal structure} on a rigid monoidal supercategory is an
even monoidal supernatural isomorphism $a\colon \id_{\mathcal{C}}\rightarrow \bid$. As in the
usual case of pivotal monoidal categories, we have the notion of left and right
quantum traces. Note that for $f\colon X\rightarrow Y$ and
$g\colon Y\rightarrow X$ homogeneous morphisms we have
\[
  \Tr^R_X(g\circ f) = (-1)^{|f||g|}\Tr^R_Y(f\circ g)
\]
and similarly for the right quantum trace.

A \emph{superpivotal structure} on a rigid monoidal supercategory is an
even monoidal supernatural isomorphism $a\colon \id_{\mathcal{C}}\rightarrow \sbid$. As in the
usual case of pivotal monoidal categories, we have the notion of left and right
quantum traces. Note that for $f\colon X\rightarrow Y$ and
$g\colon Y\rightarrow X$ homogeneous morphisms we have
\[
  \Tr^R_X(g\circ f) = (-1)^{|f|(|g|+1)}\Tr^R_Y(f\circ g)
\]
and similarly for the right quantum trace.

\subsection{Super Grothendieck ring}
\label{sec:super-groth-ring}

Let
$\mathbb{Z}_\varepsilon=\mathbb{Z}[\varepsilon]/(\varepsilon^2-1)$. The
\emph{super Grothendieck group} $\sGr(\mathcal{C})$
 of an abelian supercategory $\mathcal{C}$ is the
 $\mathbb{Z}_\varepsilon$-module generated by isomorphism classes
 $[X]$ of objects in $\mathcal{C}$ modulo the following relations: if
 $0 \rightarrow X \overset{f}{\rightarrow} Y
 \overset{g}{\rightarrow} Z \rightarrow 0$ is a short exact sequence
 with homogeneous maps, then $[Y]=\varepsilon^{|f|}[X] +
 \varepsilon^{|g|}[Z]$.

 There is a map $\sGr(\mathcal{C})/(\varepsilon-1) \rightarrow
 \Gr(\mathcal{C})$ that is clearly surjective but not
 necessarily injective.
 
 If $\mathcal{C}$ is monoidal, then tensor product in $\mathcal{C}$
 endow $\sGr(\mathcal{C})$ with a structure of an associative
 algebra. If moreover $\mathcal{C}$ is braided, the multiplication is
 commutative.

 In a superfusion category, the super Grothendieck ring is not
 necessarily a free $\mathbb{Z}_{\varepsilon}$-algebra as it may exist
 simple objects with odd automorphisms. It is nevertheless free as a $\mathbb{Z}$-algebra.

\subsection{Braided monoidal supercategory associated to a slightly degenerate braided category}
\label{sec:brmonass}

In this section, we fix a slightly degenerate pivotal braided fusion
category $\mathcal{C}$. We denote by $\e$ the unique non-unit
simple object of $\mathcal{C}'$. We fix an isomorphism $\xi\colon\e\otimes\e\rightarrow\1$.

Define a supercategory $\hat{\mathcal{C}}$ with the same objects as $\mathcal{C}$ and superspaces of morphisms being:
\[
  \Hom_{\hat{\mathcal{C}}}(X,Y)_0=\Hom_{\mathcal{C}}(X,Y) \quad \text{and} \quad
  \Hom_{\hat{\mathcal{C}}}(X,Y)_1=\Hom_{\mathcal{C}}(X,\e\otimes Y).
\]
The composition of $f\in\Hom_{\hat{\mathcal{C}}}(X,Y)$ and $g\in\Hom_{\hat{\mathcal{C}}}(Y,Z)$ is obviously defined except when $f$ and $g$ are odd morphisms. In this case, $g\circ f$ is the map given by the following composition
\[ 
\begin{tikzcd}[column sep = large]
  X \ar{r}{f} & \e\otimes Y \ar{r}{\id_{\e}\otimes
  g} &\e\otimes\e\otimes Z \ar{r}{\xi\otimes\id_{Y}} & Z.
\end{tikzcd}
\]

By a case-by-case checking, we see that the composition above is indeed associative.

The category $\hat{\mathcal{C}}$ is equipped with a tensor product. On objects, the tensor product is the same as in $\mathcal{C}$. On morphisms, we define the tensor product $f\hat{\otimes}f'$ on homogeneous morphisms $f\in\Hom_{\hat{\mathcal{C}}}(X,Y)$ and $f'\in\Hom_{\hat{\mathcal{C}}}(X',Y')$ as follows:
\begin{itemize}
\item if $f$ and $f'$ are both even, $f\hat{\otimes}f'=f\otimes f'$,
\item if $f$ is even and $f'$ odd,
  $f\hat{\otimes}f'=(c_{Y,\e}\otimes\id_{Y'})\circ f\otimes f'$,
\item if $f$ is odd and $f'$ even, $f\hat{\otimes}f'=f\otimes f'$,
\item if $f$ and $f'$ are both odd, $f\hat{\otimes}f'=(-\xi\otimes
  \id_{Y\otimes Y'})\circ(\id_{\e}\otimes c_{Y,\e}\otimes\id_{Y'})\circ f\otimes f'$.
\end{itemize}

As for the associativity, we check case-by-case, that $\hat{\otimes}$ satisfies the superinterchange law.

The braiding in $\mathcal{C}$ gives an even morphism
$c_{X,Y}\in\Hom_{\hat{\mathcal{C}}}(X\otimes Y,Y\otimes X)$.

\begin{proposition}
  Let $\mathcal{C}$ be a slightly degenerate braided category. The braiding in $\mathcal{C}$ endows $\hat{\mathcal{C}}$ with a structure of a braided supercategory.
\end{proposition}

\begin{proof}
  We have to check that for any two homogeneous morphisms
  $f\in\Hom_{\hat{\mathcal{C}}}(X,X')$ and
  $g\in\Hom_{\hat{\mathcal{C}}}(Y,Y')$ the diagram
  \[
    \begin{tikzcd}
      X\otimes Y \arrow{r}{c_{X,Y}}\ar{d}{f\hat{\otimes}g}& Y\otimes X\ar{d}{g\hat{\otimes}f}\\
      X'\otimes Y'\arrow{r}{c_{X',Y'}}& Y'\otimes X'
    \end{tikzcd}
  \]
  is $(-1)^{|f||g|}$-commutative in $\hat{\mathcal{C}}$. If $X'=X$ and
  $f=\id_X$, the commutativity follows from the naturality of the
  braiding in $\mathcal{C}$ and from one of the hexagon axioms. If
  $Y=Y'$, $g=\id_Y$ and $f$ is of even degree it is trivially
  commutative; if $f$ is of odd degree, we have to show that
  \[
    \begin{tikzcd}
      \e\otimes X'\otimes Y \ar{r}{c_{\e\otimes
        X',Y}}\ar{rd}[swap]{\id_{\e}\otimes c_{X',Y}}& Y\otimes
      \e\otimes X'\ar{d}{c_{Y,\e}\otimes \id_{X'}}\\
      & \e\otimes Y\otimes X'
    \end{tikzcd}
  \]
  is commutative in $\mathcal{C}$, which follows from one of the hexagon
  axioms and the fact that $\e$ is in the symmetric center of
  $\mathcal{C}$. The general case now follows from the superinterchange law.
\end{proof}

The category $\hat{\mathcal{C}}$ is also rigid, the evaluations and
coevaluations being the same as in $\mathcal{C}$. The dual of a
morphism of odd degree $f\in\Hom_{\hat{\mathcal{C}}}(X,Y)$ is then
given by the following map 
$f^\circledast\in\Hom_{\mathcal{C}}(Y^*,\e\otimes X^*)$
\[
  \begin{tikzcd}
    Y^*\ar{r}{\coev_X} & Y^*\otimes X \otimes X^* \ar{r}{f} &
    Y^*\otimes \e\otimes Y \otimes X^*
    \ar{r}{c_{Y^*,\e}} &\e\otimes Y^*\otimes Y
    \otimes X^* \\ &&\phantom{ Y^*\otimes \e\otimes Y \otimes X^*}\ar{r}{\ev_{Y}} &\e\otimes X^*.
  \end{tikzcd}
\]

\begin{proposition}
  If $a_{\e}=u_{\e}$ then the pivotal structure on $\mathcal{C}$ gives rise to a pivotal structure on $\hat{\mathcal{C}}$.
  If $a_{\e}=-u_{\e}$ then the pivotal structure on $\mathcal{C}$ gives rise to a superpivotal structure on $\hat{\mathcal{C}}$.
\end{proposition}

\begin{proof}
  Let $\hat{f}\in\Hom_{\hat{\mathcal{C}}}(X,Y)$ of odd degree equal to $f\in\Hom_{\mathcal{C}}(X,\e\otimes Y)$. Denote $\theta$ the twist in $\mathcal{C}$ associated to the pivotal structure. Using its naturality, we have $(\theta_{\e} \otimes \theta_{Y}) \circ f = f \circ \theta_X$ as morphisms in $\mathcal{C}$. Then, in $\hat{\mathcal{C}}$, we fins that $\theta_{\e} \theta_Y \hat{\circ}\hat{f} = \hat{f}\hat{\circ}\theta_X$, identifying $\theta_X$ and $\theta_Y$ to even morphisms in $\hat{\mathcal{C}}$ and $\theta_{\e}$ to the scalar $\pm 1$. By definition of the twist, we have $a_Z=u_Z\circ \theta_Z$ in $\mathcal{C}$ for any object $Z$ and therefore we have in $\hat{\mathcal{C}}$ 
\[
  \hat{a}_Y \hat{\circ} \hat{f} = \theta_{\e} u_Y\hat{\circ}\hat{f}\hat{\circ}\theta_X.
\]
But in $\hat{\mathcal{C}}$, $u_Y\hat{\circ}\hat{f} = \hat{f}^{**}\hat{\circ} u_X$, exactly as in a pivotal rigid monoidal category. Finally, we have
\[
  \hat{a}_Y \hat{\circ} \hat{f} = \theta_{\e} \hat{f}^{**}\hat{\circ} \hat{a}_X,
\]
which gives us a pivotal or a superpivotal structure on $\hat{\mathcal{C}}$.
\end{proof}

We end this section with the super Grothendieck ring. The quotient ring $\Gr(\mathcal{C})/([\mathbf{1}]-\dim(\e)[\e])$ can be seen as a quotient of the super Grothendieck ring of $\hat{\mathcal{C}}$ 
\[
  \Gr(\mathcal{C})/([\mathbf{1}]-\dim(\e)[\e]) \simeq \sGr(\hat{\mathcal{C}})/(1-\dim(\e)\varepsilon).
\]
The quotient ring $\sGr(\hat{\mathcal{C}})/(1-\varepsilon)$ is isomorphic to the usual Grothendieck ring of $\hat{\mathcal{C}}$ and to the quotient ring $\Gr(\mathcal{C})/([\mathbf{1}]-[\e])$. Therefore, if $\dim(\e)=-1$, we are in the following situation:
\[
\begin{tikzcd}
  & \sGr(\hat{\mathcal{C}}) \ar{dl}[swap]{\varepsilon=1}\ar{dr}{\varepsilon=-1}& \\
  \Gr(\hat{\mathcal{C}})\simeq\Gr(\mathcal{C})/([\mathbf{1}]-[\e]) & & \Gr(\mathcal{C})/([\mathbf{1}]+[\e])
\end{tikzcd}
\]
The ring $\sGr(\mathcal{C})$ is free as a $\mathbb{Z}$-module and of rank $\lvert \Irr(\mathcal{C}) \rvert$. The two quotient rings $\Gr(\mathcal{C})/([\mathbf{1}]-[\e])$ and $\Gr(\mathcal{C})/([\mathbf{1}]-[\e])$ are also free as $\mathbb{Z}$-modules and of rank $\frac{\lvert \Irr(\mathcal{C}) \rvert}{2}$. If for all simple $X,Y,Z$ we have $N_{X,Y}^ZN_{X,Y}^{\e\otimes Z}=0$ (see \cite[Question 2.19]{16fold} for a related question ; this is true if $\mathcal{C}$ is graded and the grading of $\e$ is non trivial), we remark that the structure constants of the ring $\Gr(\mathcal{C})/([\mathbf{1}]-[\e])$ relatively to a choice of a basis $J$ as in Section \ref{sec:degfus} are equal to the absolute values of the structure constants of the ring $\Gr(\mathcal{C})/([\mathbf{1}]+[\e])$ relatively to the same basis $J$.


\section{An application to Bonnafé-Rouquier asymptotic cell category}
\label{sec:appli}

In this section, we consider the example of Bonnafé and Rouquier
\cite{double_taft} and give another interpretation of their
categorification of a $\mathbb{Z}$-fusion datum. We fix an integer
$d\geq 2$ and $\zeta$ a primitive $d$-th root of unity in
$\mathbb{C}$. The algebra $D(B)$ is the algebra with generators
$K,z,E,F$ satisfying the following relations: 
\begin{align*}
  K^d&=z^d=1,\\
  E^d&=F^d=0,\\
  [z,E]&=[z,F]=[z,K]=0,\\
  KE&=\zeta EK,\\
  KF&=\zeta^{-1}FK,\\
  [E,F]&=K-zK^{-1}.
\end{align*}

This algebra is the Drinfeld double of the Taft algebra, a finite
dimensional version of the quantum enveloping algebra of a Borel of
$\mathfrak{sl}_2$. The category $D(B)\text{-}\text{mod}$ of finite
dimensional $D(B)$-modules has $d^2$ simple modules. For each integer $1\leq l \leq d$ there exists $d$ simple modules of dimension $l$ denoted by $M_{l,p}$, $p\in \mathbb{Z}/d\mathbb{Z}$ (see \cite[2.A]{double_taft} for further details).

As the algebra $D(B)$ is a braided Hopf algebra, the category
$D(B)\text{-}\text{mod}$ is a braided tensor category. We endow it with a pivotal
structure whose pivot is given by $z^{-1}K$. We denote by
$\mathcal{C}$ the semisimplification of the category $D(B)\text{-}\text{mod}$. The
simple objects in this category are the indecomposable finite dimensional $D(B)$-modules with non-zero positive and negative quantum dimension. In particular, the modules $M_{l,p}$ for $1\leq l<d$ and $p\in\mathbb{Z}/d\mathbb{Z}$ are simple in $\mathcal{C}$ and the modules $M_{d,p}$ are isomorphic to $0$.

\begin{proposition}
  \label{prop:stabtens}
  The full subcategory $\mathcal{D}$ of $\mathcal{C}$ additively generated by the simple modules $M_{l,p}$ for $1 \leq l < d$ and $p \in \mathbb{Z}/d\mathbb{Z}$ is stable by tensor product.
\end{proposition}

\begin{proof}
  As the category $\mathcal{C}$ is semi-simple and $M_{l,p}\simeq M_{l,0}\otimes M_{1,p}$, it suffices to show that, for $1\leq l \leq l' < d$ and $p\in\mathbb{Z}/p\mathbb{Z}$, the simple constituents in $\mathcal{C}$ of $M_{l,0}\otimes M_{l',p}$ are of the form $M_{n,q}$. We proceed by induction on $l$, the case $l=1$ being trivial. For $l=2$, this follows from \cite[Theorem 3.3]{double_taft}:
\[
  M_{2,0}\otimes M_{l',p}\simeq
  \begin{cases}
    M_{2,p} & \text{if } l'=1,\\
    M_{l'+1,p}\oplus M_{l'-1,p+1} & \text{if } 1<l'<d-1,\\
    M_{d-2,p+1} & \text{if } l'=d-1.
  \end{cases}
\]
Now, for $l\geq 3$ and $l'\geq l$, the module $M_{l,0}\otimes
M_{l',p}$ is a direct summand of $(M_{l,0}\oplus M_{l-2,1})\otimes
M_{l',p} \simeq M_{2,0}\otimes M_{l-1,0}\otimes M_{l',p}$. By
induction, every simple constituent of the module $M_{l-1,0}\otimes
M_{l',p}$ is of the form $M_{n,q}$ and therefore so is every simple
constituent of $(M_{l,0}\oplus M_{l-2,1})\otimes M_{l',p}$. 
\end{proof}

The category $\mathcal{D}$ is then a fusion category with $d(d-1)$ simple objects. The $S$-matrix $S^{R,R}$ has been computed by Bonnafé-Rouquier \cite[Corollary 5.5]{double_taft}:
\[
  S^{R,R}_{(l,p),(l',p')} = \frac{\zeta}{1-\zeta}\zeta^{-ll'-lp'-pl'-2pp'}(1-\zeta^{ll'}).
\]
The twist is given by
\[
  \theta_{l,p}=\zeta^{p(l+p)}.
\]
We denote by $\e$ the object $M_{d-1,1}$ which is of left and right quantum dimension $-1$ and of twist $1$. An easy calculation shows that
\[
  S^{R,R}_{(l,p),(d-1,1)}=-\dim^R(M_{l,p}),
\]
and therefore $\e$ lies in the symmetric center of the category $\mathcal{D}$. As this $S$-matrix has rank $\frac{d(d-1)}{2}$, $\e$ is the only non-trivial simple object in the symmetric center. Therefore
\[
  \mathcal{Z}_{\mathrm{sym}}(\mathcal{D})\simeq \sVect,
\]
so that the category $\mathcal{D}$ is slightly degenerate. We then compute the action of tensorisation by $\e$ on the set of simple objects.

\begin{proposition}
  In $\mathcal{D}$ we have $\e\otimes M_{l,p}\simeq M_{d-l,l+p}$.
\end{proposition}

\begin{proof}
  We proceed by induction on $l$, the case $l=1$ being trivial. For
  $l=2$, the computation is done in the proof of Proposition
  \ref{prop:stabtens}. Suppose that $l\geq 3$. We embed $\e\otimes M_{l,p}$
  in $\e\otimes (M_{l,p}\oplus M_{l-2,p+1})\simeq\e\otimes
  M_{2,0}\otimes M_{l-1,p}$ as in the proof of
  \ref{prop:stabtens}. Therefore, by induction, $\e\otimes
  M_{l,p}\otimes$ embeds in $M_{d-l,l+p}\oplus M_{d-l+2,l-1+p}$. As
  $\e\otimes M_{l-2,p+1}\simeq M_{d-l+2,p+l-1}$ we finally deduce that $M_{l,p}\otimes \e \simeq M_{d-l,l+p}$.
\end{proof}

We may choose $\{M_{l,p}\mid 0\leq p < l + p < d\}$ as a set of representatives of the action on simples given by tensorisation by $\e$.

We now compute explicitly the object $\bar{X}$, which can be reduced to the computation of the invertible object $\bar{\1}$.

The dual of $M_{l,p}$ is given by $M_{l,1-l-p}$ and therefore
\[
  s_{\1}^R(M_{l,p}^*)=\zeta^{p}\frac{\zeta^{l}-1}{\zeta-1} = s_{M_{d-1,0}}^R(M_{l,p})
\]
shows that $\bar{\1}=M_{d-1,0}$ and $\dim^R(\bar{\1})=-\zeta$. 

Following section \ref{sec:degfus}, we define the normalized $S$-matrix of the slightly degenerate pivotal braided fusion category $\mathcal{D}$ by
\[
   \tilde{\mathbf{S}}_{X,Y}=\frac{S^{R,R}_{X,Y}}{\sqrt{\dim(\bar{\1})}\sqrt{\sdim(\mathcal{D})}}.
\]
As $\sdim(\mathcal{D})=\frac{-\zeta d^2}{(1-\zeta)^2}$ (see \cite[Section 5.C]{double_taft}), we have
\[
  \tilde{\mathbf{S}}_{(l,p),(l',p')}=\frac{\zeta^{-ll'-lp'-pl'-2pp'}(\zeta^{ll'}-1)}{d}.
\]
We therefore recover the $S$ and $T$-matrices of Bonnafé-Rouquier
by the means of a slightly degenerate fusion category.

\begin{theoreme}
  The braided pivotal superfusion category $\hat{\mathcal{D}}$
  categorifies the Malle $\mathbb{Z}$-fusion datum associated to the
  non-trivial family of the cyclic group.
\end{theoreme}

This example gives us a slightly degenerate braided pivotal fusion
category which is not of the form $\mathcal{D}_0\boxtimes \sVect$ for
a nondegenerate braided pivotal fusion category $\mathcal{D}_0$.

\begin{proposition}
  If $d=2$ the category $\mathcal{D}$ is equal to $\sVect$. If $d>2$ is even, the category $\mathcal{D}$ is not of the form $\mathcal{D}_0\boxtimes \sVect$ with $\mathcal{D}_0$ nondegenerate.
\end{proposition}

\begin{proof}
  Let us suppose that $\mathcal{D}=\mathcal{D}_0\boxtimes \sVect$ with $\mathcal{D}_0$ nondegenerate. We
  consider the objects $M_{\frac{d}{2},0}$ and $\e\otimes
  M_{\frac{d}{2},0}\simeq  M_{\frac{d}{2},\frac{d}{2}}$. One of these
  two objects is in $\mathcal{D}_0$.

An easy induction shows that $(M_{2,0})^{\otimes
  2k}\otimes  M_{\frac{d}{2},0}$ has $ M_{\frac{d}{2},k}$ as a simple
constituent. Therefore, if $M_{2,0}$ belongs to $\mathcal{D}_0$,
both $M_{\frac{d}{2},0}$ and $\e\otimes  M_{\frac{d}{2},0}$
are simple objects of $\mathcal{D}_0$, which is a contradiction.

Then the object $M_{2,0}$ does not belong to $\tilde{\mathcal{D}}$ and
therefore $\e\otimes M_{2,0}$ does. But $(\e\otimes
M_{2,0})^{\otimes 2k}\otimes M_{\frac{d}{2},0} \simeq
(M_{2,0})^{\otimes 2k}\otimes  M_{\frac{d}{2},0}$ for any integer $k$,
which leads to the same contradiction.
\end{proof}


\bibliographystyle{smfalpha}
\bibliography{biblio}

\newcommand{\etalchar}[1]{$^{#1}$}
\providecommand{\bysame}{\leavevmode ---\ }
\providecommand{\og}{``}
\providecommand{\fg}{''}
\providecommand{\smfandname}{\&}
\providecommand{\smfedsname}{\'eds.}
\providecommand{\smfedname}{\'ed.}
\providecommand{\smfmastersthesisname}{M\'emoire}
\providecommand{\smfphdthesisname}{Th\`ese}
\begin{thebibliography}{BGH{\etalchar{+}}17}

\bibitem[ALW17]{superpivotal}
{\scshape D.~{Aasen}, E.~{Lake} {\normalfont \smfandname} K.~{Walker}} -- {\og
  {Fermion condensation and super pivotal categories}\fg}, \emph{ArXiv
  e-prints} (2017), \arxiv{1709.01941}.

\bibitem[BCK17]{aobc}
{\scshape J.~{Brundan}, J.~{Comes} {\normalfont \smfandname} J.~R. {Kujawa}} --
  {\og {A basis theorem for the degenerate affine oriented Brauer-Clifford
  supercategory}\fg}, \emph{ArXiv e-prints} (2017), \arxiv{1706.09999}.

\bibitem[BE17]{monoidalsupercat}
{\scshape J.~Brundan {\normalfont \smfandname} A.~P. Ellis} -- {\og Monoidal
  supercategories\fg}, \emph{Comm. Math. Phys.} \textbf{351} (2017), no.~3,
  p.~1045--1089.

\bibitem[BGH{\etalchar{+}}17]{16fold}
{\scshape P.~Bruillard, C.~Galindo, T.~Hagge, S.-H. Ng, J.~Y. Plavnik, E.~C.
  Rowell {\normalfont \smfandname} Z.~Wang} -- {\og Fermionic modular
  categories and the 16-fold way\fg}, \emph{J. Math. Phys.} \textbf{58} (2017),
  no.~4, p.~041704, 31.

\bibitem[BK01]{bakalov-kirillov}
{\scshape B.~Bakalov {\normalfont \smfandname} A.~Kirillov, Jr.} --
  \emph{Lectures on tensor categories and modular functors}, University Lecture
  Series, vol.~21, American Mathematical Society, Providence, RI, 2001.

\bibitem[BR17]{double_taft}
{\scshape C.~{Bonnaf{\'e}} {\normalfont \smfandname} R.~{Rouquier}} -- {\og {An
  asymptotic cell category for cyclic groups}\fg}, \emph{ArXiv e-prints}
  (2017), \arxiv{1708.09730}.

\bibitem[{Bru}00]{bruguiere_sln}
{\scshape A.~{Brugui\`eres}} -- {\og Tresses et structure enti\`ere sur la
  cat\'egorie des repr\'esentations de {${\rm SL}_N$} quantique\fg},
  \emph{Comm. Algebra} \textbf{28} (2000), no.~4, p.~1989--2028.

\bibitem[BV13]{bruguieres-virelizier}
{\scshape A.~Brugui\`eres {\normalfont \smfandname} A.~Virelizier} -- {\og On
  the center of fusion categories\fg}, \emph{Pacific J. Math.} \textbf{264}
  (2013), no.~1, p.~1--30.

\bibitem[EGNO15]{egno}
{\scshape P.~Etingof, S.~Gelaki, D.~Nikshych {\normalfont \smfandname}
  V.~Ostrik} -- \emph{Tensor categories}, Mathematical Surveys and Monographs,
  vol. 205, American Mathematical Society, Providence, RI, 2015.

\bibitem[ENO05]{eno}
{\scshape P.~Etingof, D.~Nikshych {\normalfont \smfandname} V.~Ostrik} -- {\og
  On fusion categories\fg}, \emph{Ann. of Math. (2)} \textbf{162} (2005),
  no.~2, p.~581--642.

\bibitem[ENO11]{eno-weakly}
{\scshape P.~Etingof, D.~Nikshych {\normalfont \smfandname} V.~Ostrik} -- {\og
  Weakly group-theoretical and solvable fusion categories\fg}, \emph{Adv.
  Math.} \textbf{226} (2011), no.~1, p.~176--205.

\bibitem[EO18]{etingof_ostrik_ss}
{\scshape P.~{Etingof} {\normalfont \smfandname} V.~{Ostrik}} -- {\og {On
  semisimplification of tensor categories}\fg}, \emph{ArXiv e-prints} (2018),
  \arxiv{1801.04409}.

\bibitem[HPT16]{henriques-penneys-tener}
{\scshape A.~Henriques, D.~Penneys {\normalfont \smfandname} J.~Tener} -- {\og
  Categorified trace for module tensor categories over braided tensor
  categories\fg}, \emph{Doc. Math.} \textbf{21} (2016), p.~1089--1149.

\bibitem[JS93]{joyal-street}
{\scshape A.~Joyal {\normalfont \smfandname} R.~Street} -- {\og Braided tensor
  categories\fg}, \emph{Adv. Math.} \textbf{102} (1993), no.~1, p.~20--78.

\bibitem[Kel05]{enriched}
{\scshape G.~M. Kelly} -- {\og Basic concepts of enriched category theory\fg},
  \emph{Repr. Theory Appl. Categ.} (2005), no.~10, p.~vi+137.

\bibitem[{Lac}18]{double_g}
{\scshape A.~{Lacabanne}} -- {\og {Drinfeld double of quantum groups, tilting
  modules and $\mathbb{Z}$-modular data associated to complex reflection
  groups}\fg}, \emph{ArXiv e-prints} (2018), \arxiv{1807.00770}.

\bibitem[Lus94]{exotic}
{\scshape G.~Lusztig} -- {\og Exotic {F}ourier transform\fg}, \emph{Duke Math.
  J.} \textbf{73} (1994), no.~1, p.~227--241, 243--248, With an appendix by
  Gunter Malle.

\bibitem[Mal95]{unip_malle}
{\scshape G.~Malle} -- {\og Unipotente {G}rade imprimitiver komplexer
  {S}piegelungsgruppen\fg}, \emph{J. Algebra} \textbf{177} (1995), no.~3,
  p.~768--826.

\bibitem[M{\"u}g03]{muger_structure}
{\scshape M.~M{\"u}ger} -- {\og On the structure of modular categories\fg},
  \emph{Proc. London Math. Soc. (3)} \textbf{87} (2003), no.~2, p.~291--308.

\bibitem[Vaf88]{vafa}
{\scshape C.~Vafa} -- {\og Toward classification of conformal theories\fg},
  \emph{Phys. Lett. B} \textbf{206} (1988), no.~3, p.~421--426.

\end{thebibliography}

\end{document}